%% file: MF-KAM-Fin.tex
\documentclass[reqno,12pt]{amsart}
\usepackage{graphics, color, amsmath, amsfonts, amssymb, amsthm, amscd}
\usepackage{cite}
\usepackage[latin1]{inputenc}

\usepackage{epsf}
\usepackage{graphicx}
\usepackage{psfrag}

\numberwithin{equation}{section}

\newtheorem{theorem}{Theorem}[section]
\newtheorem{corollary}{Corollary}[section]

\newtheorem{lemma}{Lemma}[section]
\newtheorem{proposition}{Proposition}[section]

\newtheorem{remark}{Remark}
\newtheorem{definition}{Definition}

\newtheorem{bigthm}{Theorem}   

\renewcommand{\eqref}[1]{(\ref{#1})}

\newcommand{\blue}{\textcolor{blue}}

\input{Macros}
\begin{document}

\title[Stochastic Representation of Ground States]{ 
Ground States for 
Mean Field Models with a Transverse
Component}

\author{
Dmitry Ioffe}
\address{
Faculty of Industrial Engineering\\
Technion, Haifa 3200, Israel}
\email{ieioffe\@@ie.technion.ac.il}
\thanks{}
\date{\today}
\author{Anna Levit }
\address{Department of Mathematics\\
UBC, Vancouver, B.C. V6T 1Z2, Canada  }
\email{anna.levit\@@math.ubc.ca}
\thanks{}

\vskip 0.2in
\setcounter{page}{1}
\begin{abstract}
We investigate  global logarithmic asymptotics of ground states for a family  of 
quantum mean field models. Our approach is based on a stochastic representation and a
combination of large deviation and weak KAM techniques. The spin-$\frac{1}{2}$ case is 
worked out in more detail. 
\end{abstract}
\maketitle

\section{The model and the result}
\subsection{Introduction}
\label{sub:Intro}
Stochastic representations/path integral  approach frequently provides a useful intuition and 
insight into the structure of quantum spin states. Numerous examples include 
\cite{AKN, AN94, CKP, CI, CCIL, Gi69, GUW, ILN, U}.  In this work we rely on  a path integral approach and related 
large deviations techniques, and derive {\em global} logarithmic asymptotics of ground states for a class
of quantum mean field models in transverse field. These asymtotics limits  are 
identified as weak-KAM \cite{F}  type solutions of certain Hamilton-Jacobi equations. In principle, such solutions
are not unique, and an additional refined analysis along the lines of 
\cite{DS, H88, H02} 
 is needed for recovering the correct asymptotic 
ground state. This issue is addressed in more detail for the spin $\frac12$-case.  
In particular, our results  imply logarithmic asymptotics of ground states for models with $p$-body 
interactions \cite{BS12}. 
\smallskip 

\noindent
In the case of Laplacian with periodic potential a weak KAM approach to semi-classical asymptotics  
was already employed in \cite{A}. 

\smallskip 
Our stochastic representation gives rise to a family of continuous time Markov chains on a 
simplex $\Delta_d^N$  (defined below) 
of $\frac{1}{N}\bbZ^d$.  The transition rates are enhanced by a factor of $N$, and the 
chain  moves in  a potential of the type $N F$. Ground states are Perron-Frobenius eigenfunctions
of the corresponding generators.  On the concluding stages of this work we have learned about
the series of papers \cite{KR1, KR2, KR3}.   The models we consider here essentially fall into a much
more general framework studied in these works.  The authors of \cite{KR1, KR2, KR3}
extend an analysis of Schr\"{o}dinger operators \cite{DS, H88, H02}  on $\bbR^d$ to 
lattice operators on $\epsilon\bbZ^d$, 
and they 
 develop powerful 
techniques, which go well beyond the scope of our work, and 
 which enable a complete asymptotic expansion of low lying eigenvalues and eigenfunctions
in neighbourhoods of potential wells.   
\smallskip 

\noindent
The paper is organized as follows:  The class of models is described in Subsection~\ref{sub:Models}, 
and the results are formulated in Subsection~\ref{sub:Results}.  Main steps of our approach are
explained  in Section~\ref{sec:Structure}, whereas some of the proofs are relegated to 
Section~\ref{sec:Proofs}.  The spin-$\frac12$ case is  studied in Section~\ref{sec:S12}.  Finally, 
in the Appendix, we establish the required properties of the Lagrangian $\cL_0$ 
in \eqref{eq:HLnot}   and, accordingly, 
the required regularity properties of local minimizers.  
\subsection{Class of Models} 
\label{sub:Models}
Let $\bbX$ be a $d$-dimensional complex Hilbert space.
For the rest of the paper we fix an orthonormal basis 
$\lbr \ket{\alpha}\rbr_{\alpha\in {\cA}} $ of 
$\bbX$.   We refer to the set $\cA$ of cardinality $d$ as the set of classical labels.
Denote projections $P_\alpha\df \ket{\alpha}\bra{\alpha}$. 
The induced basis of $\bbX_N = \otimes_1^N\bbX $ is 
\[
 \ket{\ualpha} = \ket{\alpha_1}\otimes\dots\otimes \ket{\alpha_N}\quad \alpha_1, \dots, 
\alpha_N\in\cA .
\]
The corresponding lifting of the projection operator acting on $i$-th component is 
$P^i_\alpha = I\otimes \dots I\otimes  P_\alpha\otimes I\otimes\dots\otimes I$. 
For $\alpha\in\cA$ set $M^N_\alpha =\frac{1}{N}\sum_i  P^i_\alpha$. 
Let $\underline{M}^N$ be the $d$-dimensional vector with operator entries $M^N_\alpha$.

We are ready to define the Hamiltonian 
$\cH_N$ which acts on $\bbX_N$, 
\be
\label{eq:HN}
-\cH_N = N  F \lb 
\underline{M}^N
\rb + \sum_i B_i . 
\ee
Above,  $B_i$-s are copies of a Hermitian matrix $B$ on $\bbX$, $B_i$ acts on the $i$-th
component of $\ket{\ualpha}$. 

We assume:


\paragraph{\bf A1.} $F$  is a real polynomial of finite degree. 
\smallskip 

\noindent
Let $\Delta_d$ be the simplex, $\Delta_d = \lbr \um\in\bbR^d_+ \sth
\sum m_i = 1\rbr$. In the sequel we shall write ${\rm int}\lb \Delta_d\rb$ 
for the relative interiour 
of $\Delta_d$. 
Accordingly,
 $\partial\Delta_d \df \Delta_d\setminus{\rm int}\lb \Delta_d\rb $,

Given $\um\in \Delta_d$ and a basis vector 
$\ualpha\in\cA^N$ let us say
that $\ualpha\sim\um$, or, equivalently, $\um = \um  (\ualpha )$,  if 
\be 
\label{eq:malpha}
m_\alpha = \frac{\# \lbr i\sth \alpha_i = \alpha\rbr}{N} \Leftrightarrow {M}^N_\alpha 
\ket{\ualpha} = m_\alpha\ket{\ualpha}  ,
\ee
for all $\alpha\in\cA$.  
Define 
$\Delta_d^N = \Delta_d\cap \frac1{N}\bbZ^d$. In other words, 
$\um\in  \Delta_d^N $ iff there exists a compatible $\ualpha\in\cA^N$.
In the above notation:
\be  
\label{eq:Psium}
 F \lb \underline{M}^N 
 \rb
\ket{\ualpha} 
= F (\um (\ualpha ) )\ket{\ualpha} .
\ee
\paragraph{\bf A2.} The transverse field $B$ is stochastic: For any $\alpha ,\beta\in \cA$ , 
\be  
\label{eq:rates}
\lambda_{\alpha \beta } = \lambda _{\beta \alpha}\df \bra{\alpha}B\ket{\beta} \geq 0.
\ee
Furthermore, $\lambda$ is an irreducible kernel on $\cA$. 
Without loss of generality we shall assume that $\lambda\equiv 0$ on the diagonal.

\subsection{An Example: Spin-$\sfs$  Models} 
The relation between the dimension $d$ of $\bbX$ and the half-integer spin 
$\sfs$ is $d=2\sfs +1$. The set of classical labels is
\[
 \cA =\lbr -\sfs, -\sfs +1, \dots, \sfs\rbr .
\]
The stochastic operators 
are $B_i =\lambda  \sfS_i^x$. 
$\lambda\geq 0$ is  the strength of the transverse field. 
Altogether, the Hamiltonian is
\be  
\label{eq:spins}
 -\cH_N = N F\lb
M^N_{-\sfs }, M^N_{-\sfs + 1}, \dots , M^N_{\sfs}
\rb  +\lambda \sum_{i}\sfS_i^x .
\ee
For instance, the case of $p$-body ferromagnetic interaction corresponds to
\begin{equation}
\label{eq:p-body}
 F\lb M^N_{-\sfs }, M^N_{-\sfs - 1}, \dots , M^N_{\sfs}\rb = 
\lb \sum_\alpha \alpha M^N_\alpha\rb^p = \lb \frac{1}{N}\sum_i \sfS^z_i\rb^p .
\end{equation}
The operators 
$\sfS^x$ act (under convention that
$\ket{\sfs +1} = \ket{-\sfs - 1}= 0$)  on $\bbX$ as
\be 
\label{eq:SAction}
\sfS^x\ket{\alpha} = \frac{1}{2}\sqrt{\sfs (\sfs +1 ) - \alpha (\alpha-1 )}\ket{\alpha -1} +
\frac{1}{2}\sqrt{\sfs (\sfs +1 ) - \alpha (\alpha+1 )}\ket{\alpha +1}
\ee
Consequently, the jump rates $\lambda_{\alpha\beta}$ are given by 
\be
\label{eq:lambdaS}
\lambda_{\alpha \beta } = 
\begin{cases}
\frac{\lambda}{2}\sqrt{\sfs (\sfs +1 ) - \alpha\beta}, \quad &\text{if $|\alpha -\beta |=1$}\\
0, &\text{otherwise}. 
\end{cases}
\ee
\subsection{The Result}
\label{sub:Results}
In order to develop an  asymptotic description of finite volume ground states 
we need to introduce some additional notation: For
$\um\in \Delta_d^N$ set
\[
c_N (\um ) = {N\choose N\um} = \frac{N!}{\prod (Nm_\alpha )!} . 
\]
The vectors $\ket{\um}\in\bbX_N$, 
\[
 \ket{\um} \df  \frac{1}{\sqrt{c_N (\um )}}\sum_{\ualpha\sim \um}\ket{\ualpha} 
\]
are normalized and orthogonal for different $\um\in \Delta_d^N $.

By Perron-Frobenius theorem and Lemma \ref{lem:HNLumping} below 
the ground state of  $\cH_N$ is fully symmetrized, that is of the form
\be 
\label{eq:GSHN}
\ket{h_N} = \sum_{\um\in\Delta_d^N} h_N (\um )\ket{\um} , 
\ee
and  $h_N (\um ) >0$ for every $\um\in\Delta_d^N$ (see Subsection~\ref{sub:Spec}).  Let 
us represent 
\be  
\label{eq:hN}
h_N (\um ) = {\rm e}^{-N\psi_N (\um )} .
\ee
It would be convenient to identify $\psi_N$ with its linear interpolation (which is an element of 
the space of continuous functions $\sfC (\Delta_d )$). 

Next introduce:
\be
\label{eq:HLnot}
\begin{split}
&\cH_0 \lb \um ,\utheta\rb = \sum_{\alpha\beta}\sqrt{m_\alpha m_\beta}\, \lambda_{\alpha\beta}
\lb {\rm cosh} (\theta_\beta - \theta_\alpha ) -1\rb 
 \\
&
\cL_0 (\um , \uv ) = \sup_{\utheta} \lbr (\uv , \utheta)  - \cH_0 (\um , \utheta )\rbr 
\end{split}
\ee
For $\um\in\Delta_d$ 
define 
\be
\label{eq:VE}
V(\um ) = 
\frac12
\sum_{\alpha ,\beta }\lambda_{\alpha\beta} 
\lb 
\sqrt{m_\beta} 
- \sqrt{m_\alpha}\rb^2 
- F (\um ) .
\ee
Finally set $\lambda_\alpha  = \sum_\beta \lambda_{\alpha \beta }$, 
\be
\label{eq:HL}
\begin{split}
\cH  \lb \um ,\utheta\rb  &= \cH_0  \lb \um ,\utheta\rb - V(\um ) \\
&= 
\sum_{\alpha\beta}\sqrt{m_\alpha m_\beta}\lambda_{\alpha\beta}
{\rm cosh} (\theta_\beta - \theta_\alpha ) - \sum_\alpha m_\alpha\lambda_\alpha + F(\um ) ,
\\
&\qquad {\rm  and}\\
\cL  (\um , \uv ) &= \cL_0 (\um , \uv ) + V (\um ) .
\end{split}
\ee
In \eqref{eq:HL} above $(\cdot ,\cdot )$ is the usual scalar product on $\bbR^d$. 
\smallskip 
%
\begin{bigthm}
 \label{thm:A}
Let $E_N^1$ be the bottom eigenvalue of $\cH_N$. 
Set $\lambda = \sum_\alpha \lambda_\alpha $. 
Then the limit
\be 
\label{eq:e1}
-\lambda + \sfr_1 \df \lim_{N\to\infty}\frac{E_N^1}{N} 
\ee
exists. Moreover,
\be 
\label{eq:C1}
\sfr_1 = \min_{\um } V (\um )   .
\ee
Furthermore, the sequence $\lbr\psi_N\rbr$ is precompact in $\sfC(\Delta_d )$. Any subsequential limit $\psi$ 
satisfies: 
 For any $T\geq 0$ and any $\um\in \Delta_d$,
\be 
\label{eq:C2}
\psi (\um ) = \inf_{\gamma\sth \gamma (T) = \um}\lbr 
\psi (\gamma (0)) + \int_0^T \cL \lb \gamma (t) ,\gamma^\prime (t )\rb\dd t - T\sfr_1
\rbr ,
\ee 
where the infimum above is over all absolutely continuous curves $\gamma : [0,T]\to 
\Delta_d$.   Moreover, the set $\cM_\psi$ of all local minima of $\psi$ is a subset of 
$\cM_V\df {\rm argmin} (V)\subset {\rm int}(\Delta_d )$. 
\end{bigthm}
\begin{remark}
Hamiltonians  $\cH_0 ,\cH$ are invariant under the shifts $\utheta\mapsto \utheta +c {\bf 1}$,
 and, as a result, the Lagrangians $\cL_0 ,\cL$ are infinite whenever $(\uv ,{\bf 1})\neq 0$. 
Also, $\cL_0 (\um , 0) = 0 = \min \cL_0 (\um ,\uv )$. Consequently, 
$\cL (\um ,0 ) = V(\um ) = -\cH (\um ,0 )$, and \eqref{eq:C1} could be 
rewritten as 
\be  
\label{eq:C11}
\sfr_1 = \min_{\um ,\uv  } \cL (\um ,\uv ) = - \max_{\um} \cH (\um ,0 ).
\ee 
\end{remark}
\begin{remark}
Either of  \eqref{eq:C1} and \eqref{eq:C2}   unambiguously characterizes $\sfr_1$, but not $\psi$. 
As we shall explain in the sequel, if $\psi$ satisfies \eqref{eq:C2}, then 
 the weak KAM theory of Fathi \cite{F}  implies that $\psi$ is 
 a viscosity solution  (see Subsection~\ref{sub:VS} for the precise statement) 
on ${\rm int}\lb \Delta_d\rb$ of  the Hamilton-Jacobi equation
\be  
\label{eq:C3}
\cH \lb \um ,\nabla \psi  (\um )\rb = -\sfr_1 . 
\ee
\end{remark}
Note that since $\psi$ is a function on $\Delta_d$, the 
 gradients $\nabla\psi$  lie in the subspace 
\be 
\label{eq:Rn0}
\bbR^d_0 = \lbr \uv  :  (\uv, {\bf 1}) =0\rbr .
\ee 
In general  there might be many viscosity solutions of \eqref{eq:C3} which comply with the conclusions of
Theorem~\ref{thm:A}.  The solutions which are subsequential limits of $\lbr \psi_N\rbr$ 
are  called {\em admissible}. 
Although we expect uniqueness of {global}  admissible solutions for a large class of models,
our approach does not offer a procedure for selecting the latter.   
The viscosity setup is important - at least for a large class of symmetric potentials the global 
admissible solutions are not smooth and develop shocks.  
 A proper selection procedure should be related 
to a more refined analysis of the low lying spectra of 
$\cH_N$,  As it was mentioned in the Introduction sharp asymptotics of  eigenvalues and 
eigenfunctions in  vicinity of potential wells were derived in a much more general context in 
\cite{KR1, KR2, KR3}.  In particular, it is explained therein how such asymptotics are 
related to (smooth) local solutions of \eqref{eq:C3}. 
Implications of these results for a characterization of {\em global} admissible solutions 
is beyond the scope of this work  and hopefully  shall be addressed in full 
generality elsewhere.  In the concluding Section~\ref{sec:S12} we work out a 
particular case of  spin-$\frac12$ models..  

\section{Structure of the theory}
\label{sec:Structure} 
\subsection{Spectrum of $\cH_N$}
\label{sub:Spec}
Let $\bbX_N^{\sfs}$ be the sub-space of $\bbX_N$ which consist of those vectors $\ket{b}$ which do not vanish
under symmetrization. Namely, $\ket{b}  =\sum_{\ualpha} a_{\ualpha}\ket{\ualpha} \in \bbX_N^{\sfs}\setminus 0$,
if $\sum_\pi  a_{\pi \ualpha} \neq 0$ for some $\ualpha\in\cA^N$, where $\pi$ is a permutation 
of $\lbr 1, \dots ,N\rbr$ with $\pi\ualpha (i) = \ualpha (\pi_i )$. The sub-space $\bbX_N^{\sfs}$ is 
invariant for the Hamiltonian $\cH_N$.  The ground state of $\cH_N$ always belongs to $\bbX_N^{\sfs}$. 
For the rest of the paper we shall work with the restriction of $\cH_N$ to $\bbX_N^{\sfs}$. 
\smallskip 

  All eigenfunctions of $\cH_N$  (restricted to $\bbX_N^{\sfs}$) have 
mean-field representatives:
\begin{lemma}
 \label{lem:HNLumping}
If $E_N$ is an eigenvalue of $\cH_N$, then there exists a function $h_N$
on $\Delta_d^N$ such that $\ket{h_N}\df \sum_{\um\in \Delta_d^N}h_N (\um )\, \ket{\um}$ is 
a corresponding eigenfunction:
\be  
\label{eq:HNLumping}
\cH_N 
\ket{h_N}
= 
E_N \ket{h_N}.
\ee
\end{lemma}

\noindent
{\em Proof.}
Let $E_N$ be an eigenvalue of $\cH_N$.
 Let $\ket{b_N}=\sum_{\ualpha\in\cA^N}a_{\ualpha}\ket{\ualpha} \in\bbX_N^{\sfs}$ be an 
eigenfunction corresponding to the eigenvalue $E_N$. 
Let $C (\ualpha ,\ubeta) = \bra{\ubeta}\hat B\ket{\ualpha}$ be the matrix elements 
of $\hat{B}\df \sum_i B_i$. Thus, $\hat B\ket{\ualpha} = \sum_{\ubeta } C(\ualpha ,\ubeta )\ket{\ubeta}$. 
The eigenfunction equation is recorded as: $\forall\,\ubeta$
\[
 \sum_{\ualpha } a_{\ualpha } C(\ualpha ,\ubeta ) = \lb - E_N - F (\um )\rb a_{\ubeta } .
\]
Note that $C (\ualpha , \ubeta ) = C (\pi\ualpha ,\pi\ubeta )$. Consequently, since in addition $\um (\ubeta ) = \um (\pi \ubeta )$, 
\[
 \sum_{\ualpha } a_{\pi \ualpha } C(\ualpha ,\ubeta ) = \lb - E_N - F (\um )\rb a_{\pi \ubeta } .
\]
Therefore, $\ket{\pi b_N}\df \sum_{\ualpha}a_{\pi\ualpha }\ket{\ualpha }$ is also an eigenfunction.
Since the sum  $\sum_{\pi} a_{\pi \ualpha}$ does not change if we permute the entries of 
$\ualpha$, and since, by assumption $\ket{b_N}\in\bbX_N^0$, the claim follows
with $\ket{h_N} = \sum_{\pi}\ket{\pi b_N}$. 
\qed
\medskip

\subsection{Stochastic Representation}
\label{sub:SR}
  Let $\alpha (t)$ be the continuous time 
Markov chain on $\cA$ with jump rates $\lambda_{\alpha \beta }$. $\bbP^N_{\ualpha}$
is the path measure for $N$ independent copies of such chain starting from
$\ualpha$. 
Then the following representation of the  entries of the density matrix holds 
\cite{AN94, ILN}: For any $T \geq 0$ and any $\ualpha ,\ubeta\in \cA^N$
\be 
\label{eq:DMalpha}
{\rm e}^{-N \lambda T }
\bra{\ubeta} {\rm e}^{-T\cH_N}\ket{\ualpha} 
=\bbE^N_{\ualpha}  {\rm exp}\lbr N \int_0^T  F  (\um (t) )\dd t\rbr \1_{\lbr \ualpha ( T ) 
= \ubeta\rbr }.
\ee
Above $\um (t) = \um (\ualpha (t ))$. 
\medskip 

\subsection{Mean Field Lumping}
\label{sub:MFL}
The process $\um_N (t) = \um (t ) = \um (\ualpha (t))$ is a continuous time Markov chain
on $\Delta_d^N$ with the generator
\be  
\label{eq:umGenerator}
\cG_N f (\um ) = N \sum_{\alpha ,\beta }m_\alpha \lambda_{\alpha \beta }
\lb f\lb \um +\frac{\delta_\beta - \delta_\alpha}{N} \rb- f (\um)\rb . 
\ee 
It is reversible with respect to the measure
\be 
\label{eq:muN}
\mu_N (\um )\df \frac{c_N (\um )}{d^N} .
\ee
Summing up in \eqref{eq:DMalpha}, 
\be  
\label{eq:DMm}
{\rm e}^{-N \lambda T}
\bra{\um^\prime} {\rm e}^{-T\cH_N}\ket{\um} =
\sqrt{\frac{\mu_N (\um ) }{\mu_N (\um^\prime )}}
\bbE_{\um}^N  
 {\rm exp}\lbr N \int_0^T  F  (\um (t) )\dd t\rbr \1_{\lbr \um ( T ) 
= \um^\prime \rbr }, 
\ee
for every $T\geq 0$ and every $\um , \um^\prime\in \Delta_d^N$. 

Using Girsanov's formula one can rewrite \eqref{eq:DMm} in a variety of ways for different modifications
of the jump rates in \eqref{eq:umGenerator}. Namely, let $g$ be a positive function on $\Delta_d^N$. Consider
the modified rates
\[
 \lambda_{\alpha \beta}^{N, g}= \lambda_{\alpha \beta}^{N, g} (\um )  
= \frac{1}{g (\um )}N m_\alpha \lambda_{\alpha \beta }
g \lb\um +(\delta_\beta - \delta_\alpha)/N\rb .
\]
Let $\bbP_{\um}^{N ,g}$ be the corresponding path measure. Then, the right hand side of 
\eqref{eq:DMm} reads as 
\be
\label{eq:rhsg} 
\sqrt{\frac{\mu_N (\um ) g (\um )^2 }{\mu_N (\ump
)g (\um^\prime )^2}}
\bbE_{\um}^{N ,g}  
 {\rm exp}\lbr N \int_0^T  F_g  (\um (t) )\dd t\rbr \1_{\lbr \um ( T ) 
= \um^\prime \rbr }, 
\ee
where
\be 
\label{eq:Fg}
 F_g (\um ) = \sum_{\alpha \beta} \lambda_{\alpha\beta} \, m_\alpha
\lb \frac{g \lb \um +(\delta_\beta - \delta_\alpha)/N\rb}{ g (\um )} -1\rb + F (\um ).
\ee
A self-suggesting choice is 
\be  
\label{eq:gchoice}
g (\um ) = \frac{1}{\sqrt{\mu_N (\um )}}\ \Rightarrow \ 
 \lambda_{\alpha \beta}^{N, g} = N\sqrt{m_\alpha (m_\beta + 1/N )}\, \lambda_{\alpha\beta} .
\ee
For the rest of the paper we fix $g$ as in \eqref{eq:gchoice}. The corresponding generator 
\be  
\label{eq:umGeneratorg}
 \cG_N^g f (\um ) = \sum_{\alpha ,\beta } \lambda_{\alpha \beta }^{N ,g}
\lb f\lb \um +\frac{\delta_\beta - \delta_\alpha}{N} \rb- f (\um)\rb . 
\ee
is reversible with respect to the uniform measure on $\Delta_d^N$.  The function $F_g $ in 
\eqref{eq:Fg} equals to
\be  
\label{eq:FgV}
F_g (\um ) = - V (\um ) + \Xi_N (\um ), 
\ee
where $V$ is precisely the function defined in \eqref{eq:VE}, and  the correction 
\be
\label{eq:XiN} 
\Xi_N (\um ) = \sum_{\alpha\beta} \lambda_{\alpha \beta }\sqrt{m_\alpha } \lb 
\sqrt{m_\beta +1/N} - \sqrt{m_\beta } \rb .
\ee
All together, \eqref{eq:DMm} reads as
\be 
\label{eq:DMmg}
{\rm e}^{-N \lambda T}
\bra{\um^\prime} {\rm e}^{-T\cH_N}\ket{\um} =
\bbE_{\um}^{N,g}  
 {\rm exp}\lbr  - N \int_0^T \lb V   - \Xi_N\rb(\um (t) )  \dd t\rbr \1_{\lbr \um ( T ) 
= \um^\prime \rbr }, 
\ee
As we shall see below it happens to be convenient to work simultaneously 
with both representations
\eqref{eq:DMm} and \eqref{eq:DMmg}. 

Note that an immediate consequence of \eqref{eq:DMmg} is:
\begin{lemma}
\label{lem:ENLumping} $E_N$ is an eigenvalue of $\cH_N$ with $\ket{u_N} = \sum u_N
(\um )\ket{\um }$ being   the corresponding normalized eigenfunction if and only if 
$u_N$ is also an eigenfunction
of $\cS_N\df \cG_N^g + N F_g = \cG_N^g  - N (V - \Xi_N )$ with 
\be 
\label{eq:RN}
-R_N \df - (N\lambda + E_N )
\ee
 being the 
corresponding eigenvalue.
\end{lemma}

\subsection{The Eigenfunction Equation}
\label{sub:EE}
Assumption {\bf A.2}
 and Perron-Frobenius theorem imply that $\cH_N$ has a non-degenerate ground
state $\ket{h_N} = \sum_{\um } h_N (\um )\ket{\um}$ with strictly
positive entries $h_N (\um  )>0$.  By Lemma~\ref{lem:ENLumping}, 
$h_N (\um )$  is the principal  eigenfunction 
of $\cG_N^g + NF_g (\um )$ with the corresponding top eigenvalue $-R_N^1 = 
- (N\lambda + E_N^1)$.  The corresponding eigenfunction  equation is: For any $T >0$,
\be 
\label{eq:EigenEquation}
 h_N (\um ) = \bbE_{\um }^{N ,g} 
{\rm exp}\lbr - N \int_0^T \lb V -\Xi_N\rb (\um (t) )\dd t + T R_N^1 \rbr h_N  (\um (T)) ,
\ee
By reversibility, 

\begin{lemma}
\label{lem:hatgN}
Functions $\lbr h_N\rbr$ satisfy: For every  $T\geq 0$, $N$ and $\um\in\Delta_d^N$
\be  
\label{eq:hatgN}
h_N (\um ) = \sum_{\um^\prime}
h_N (\um^\prime ) \bbE_{\um^\prime}^{N ,g}
{\rm exp}\lbr- N \int_0^T  \lb V-\Xi_N\rb  (\um (t) )\dd t +TR_N^1 \rbr\1_{\lbr \um (T )=\um\rbr} .
\ee 
\end{lemma}
\medskip 

\subsection{Compactness and Large Deviations}
\label{sub:CLD}
For $\um ,\ump\in\Delta_d^N$ define
\be  
Z_T^{N,g} (\um^\prime , \um )\df 
\frac{1}{N}\log 
\bbE_{\um^\prime}^{N,g }
{\rm exp}\lbr - N \int_0^T  (V- \Xi_N ) (\um (t) )\dd t \rbr\1_{\lbr \um (T )=\um\rbr},
\label{eq:ZTNg}
\ee
In the sequel we shall identify $Z_T^{N,g} (\cdot, \cdot )$ with its  continuous 
interpolation on 
$\Delta_d\times\Delta_d$. 

Let $\cA\cC_T$ be the family of all absolutely continuous trajectories 
$\gamma :[0,T]\mapsto \Delta_d$.  For $\um , \ump\in\Delta_d$ define
\be   
\label{eq:ZT}
Z_T^g  (\um^\prime , \um )\df 
- \inftwo{\gamma (0 )=\um^\prime , \gamma (T) = \um}{\gamma\in\cA\cC_T} 
\int_0^T \cL (\gamma (t) , \gamma^\prime (t ))\dd t ,
\ee
where the Lagrangian $\cL$ was defined in \eqref{eq:HL}.  
\begin{theorem}
 \label{thm:LD}
For all $T$ sufficiently large the sequence of functions $\lbr Z_T^{N,g}\rbr$ 
is equi-continuous on 
$\Delta_d\times\Delta_d$
 and uniformly locally Lipschitz on ${\rm int}\lb \Delta_d\times\Delta_d\rb $. 
Furthermore, for all $T$ sufficiently large, 
\be 
\label{eq:LD}
\lim_{N\to\infty}
\left| 
Z_T^{N, g} (\um^\prime , \um ) 
- 
Z_T^g  (\um^\prime , \um )
\right| = 0
\ee
simultaneously for all $\um, \ump\in \Delta_d$.  
\end{theorem}
Note that the equi-continuity   of $\lbr Z_T^{N, g}\rbr$ in 
Theorem~\ref{thm:LD} implies 
that the convergence in \eqref{eq:LD} is actually 
uniform. Consequently, $Z_T^g (\cdot ,\cdot )$ is continuous 
on $\Delta_d\times\Delta_d$
and locally Lipschitz on 
${\rm int}\lb \Delta_d\times\Delta_d\rb $. 

Theorem~\ref{thm:LD} is a somewhat standard statement. Its proof will be sketched in Subsection~\ref{sub:LD}.
\medskip 

\subsection{Lax-Oleinik Semigroup and Weak KAM}
\label{sub:LOSWKAM}
Recall the representation of  the leading eigenfunction $h_N (\um ) = {\rm e}^{-N \psi_N (\um )}$.  
 In the sequel we shall identify 
$\psi_N$ with its (continuous ) interpolation on $\Delta_d$; 
$\psi_N  \in \sfC (\Delta_d )$. 
\begin{theorem}
\label{thm:LDUniform}
 The sequence of numbers $R_N^1/N$ is bounded in $\bbR$.  The sequence of functions 
$\lbr \psi_N\rbr$ is 
equi-continuous on $\Delta_d$ and uniformly 
locally Lipschitz on ${\rm int}\lb\Delta_d\rb$. 
\end{theorem}
\begin{proof}
Since $R_N^1$ is the Perron-Frobenius eigenvalue, 
 \[
\frac{R_N^1	}{N}=-\frac{1}{N}\lim_{T\rightarrow\infty}\frac{1}{T}\log \bbE_{\um'}^{N, g}\exp{\lbr N\int_0^TF_g  (\um (s) )\dd s\rbr}, 
\]
which is bounded since $F_g$ is bounded on $\Delta_d$. 
On the other hand, by \eqref{eq:hatgN}, the equi-continuity 
and the uniform local Lipschitz property
of $\lbr \psi_N\rbr$ is inherited from the 
corresponding properties  of $\lbr Z_T^{N, g}\rbr$.
\end{proof}

\noindent \textbf{Proof of \eqref{eq:e1} and \eqref{eq:C2} of Theorem~\ref{thm:A}:} Theorems ~\ref{thm:LD},~\ref{thm:LDUniform} and Lemma~\ref{lem:hatgN} imply that any accumulation point
 $(\sfr , \psi )\in \bbR\times \sfC (\Delta_d )$ of
the sequence $\lbr \frac{1}{N} R_N^1 , \psi_N \rbr$ satisfies:
$\psi$ is locally Lipschitz on ${\rm int}\lb\Delta_d\rb$ and 
\be  
\label{eq:KAM}
\psi (\um ) = \inf_{\gamma (T) = \um}\lbr \psi\lb \gamma (0)\rb
+ \int_0^T \cL (\gamma (t) , \gamma^\prime (t ))\dd t \rbr - T \sfr 
\df \cU_T \psi (\um ) - T\sfr .
\ee
for every $T\geq 0$ and each $\um\in \Delta_d$.   
Accumulation points of $\psi_N$ are called {\em admissible solutions} 
of \eqref{eq:KAM}. 
Since $\cU_T$ is non-expanding on $\sfC (\Delta_d)$,   
 validity of equation 
\eqref{eq:KAM} unambiguously determines $\sfr$, which implies  that 
the limit $\sfr_1\df \lim_{N\to\infty}\frac{R_N^1}{N}$ indeed exists.  
In view of \eqref{eq:RN} we, therefore, 
 have established \eqref{eq:e1} and \eqref{eq:C2} of Theorem~\ref{thm:A}.\qed 
\medskip 

\subsection{Viscosity Solutions} 
\label{sub:VS} Recall the definition \eqref{eq:HL}
 of $\cH (\um ,\utheta )$. For $\sfr_1$ defined as above consider the 
Hamilton-Jacobi equation
\be 
\label{eq:HJ}
\cH (\um , \nabla\psi (\um )) = -\sfr_1 .
\ee
\begin{definition}
 For $\um\in \Delta_d$ and $\psi\in\sfC (\Delta_d )$ define
lower and upper sub-differentials
\be  
\label{eq:subdif}
D_-\psi (\um ) =
\lbr\xi\in\bbR_{0}^n\sth \liminf_{\um^\prime \to\um}\frac{\psi (\um^\prime ) -
\psi (\um ) - (\xi, \um^\prime -\um )}{|\um^\prime -\um |}\geq 0\rbr ,
\ee
and, similarly, $D_+\psi (\um )$ with $\liminf$ changed to $\limsup$ and the sign of the 
inequality flipped. 

A locally Lipschitz  function $\psi$ is said to be a viscosity supersolution of 
\eqref{eq:HJ} at $\um$ if $\cH (\um ,\xi )\geq -\sfr_1$ for every $\xi\in D_-\psi (\um )$. 
Similarly, it is said to be a viscosity subsolution of 
\eqref{eq:HJ} at $\um$ if $\cH (\um ,\xi )\leq -\sfr_1$ for every $\xi\in D_+\psi (\um )$. 
$\psi$ is viscosity solution of \eqref{eq:HJ}  at $\um$ if it is both sub and super viscosity
solution. 
\end{definition}
\begin{theorem}
\label{thm:constraint}
If $\psi$ is a weak-KAM solution (of \eqref{eq:KAM} with $\sfr = \sfr_1$), then
it is a viscosity solution of \eqref{eq:HJ} on ${\rm int}\lb \Delta_d\rb$.
\end{theorem}
The proof of Theorem~\ref{thm:constraint} is relegated to 
Subsection~\ref{sub:PConstraint}.
\medskip 

\subsection{Minima of $ \psi $} 
\label{sub:MVarPhi}
\begin{theorem}
 \label{thm:ent}
Let $\psi$ be a weak KAM solution of \eqref{eq:KAM}. Then all { local} minima  of 
$\psi $ lie in the interiour ${\rm int}(\Delta_d)$. 
\end{theorem}
Theorem~\ref{thm:ent} will be proved in Subsection~\ref{sub:PMVArPhi}
\medskip

\subsection{Stochastic Representation of the Ground State}  
\label{sub:SRGS}
The eigenfunction equation \eqref{eq:EigenEquation}
defines a Markovian semi-group
\be 
\label{eq:hatSG} 
\widehat{\bbE}^N_T f (\um ) = 
\frac{1}{h_N (\um )}
\bbE_{\um}^{N ,g} 
{\rm exp}\lbr N \int_0^T  F_g (\um (t) )\dd t + T R_N^1 \rbr h_N (\um (T)) f (\um (T )) .
\ee
This corresponds to continuous time Markov chain with the generator
\be
\label{eq:umGeneratorDoob}
\widehat\cG_{N}^{g} f (\um ) = \frac{1}{h_N (\um )} 
\sum_{\alpha ,\beta } \lambda_{\alpha \beta }^{N ,g}
h_N \lb \um +\frac{\delta_\beta - \delta_\alpha }{N}\rb 
\lb f \lb \um +\frac{\delta_\beta - \delta_\alpha }{N}\rb- f (\um)\rb . 
\ee 
In the sequel we shall refer to $\widehat{\cG}_{N}^{g}$ as to the generator of the 
ground state chain.  
\begin{lemma}
\label{lem:GS} 
The generator 
$\widehat{\cG}_N^g$ is reversible with respect to
the probability  measure $\nu_N (\um )\df h_N^2 (\um )  = {\rm e}^{-2N\psi_N (\um )}$. Furthermore,
 $E_N$ is an eigenvalue of $\cH_N$ if and only 
if $E_N^1 - E_N$ is an eigenvalue of $\widehat{\cG}_N^g$. 
\end{lemma}
It is straightforward to check that $\widehat{\cG}_N^g$ satisfies
the detailed balance condition with respect to $\nu_N$. 
It is  equally straightforward to see  
from \eqref{eq:hatSG}  that $g_N$ is an eigenfunction of $\cG_N^g + NF_g$, 
and hence by Lemma~\ref{lem:ENLumping} of $\cH_N$,  if and only if 
$g_N/h_N$ is an eigenfunction of $\widehat{\cG}_N^g$.  

\section{Proofs}
\label{sec:Proofs} 
\subsection{Proof of Theorem~\ref{thm:LD}}
\label{sub:LD}
Consider the family of processes $\lbr \um (\cdot ) = \um_N (\cdot )\rbr$ with generator 
$\cG_N^g$ defined in \eqref{eq:umGeneratorg}. We shall identify $\um_N$ with its linear 
interpolation. For each $T >0$, the family $\lbr \um_N (\cdot )\rbr$ is exponentially tight
on $\sfC_{0,T}\lb\Delta_d \rb$. 

Recall the definition of $\cL_0$ in \eqref{eq:HLnot}.
One can follow the approach of \cite{DZ} and to combine 
the Large Deviation Principle  for projective limits \cite{DG} with the inverse 
contraction principle of \cite{I91}
in order to conclude:
\begin{lemma}
\label{lem:LD}
For  each $T >0$ and every initial condition $\um\in\Delta_d$ (where for each $N$ we identify 
$\um$ with its discretization $\lfloor N\um\rfloor/N \in \Delta_d^N$) 
the family of processes
$\lbr\um(t)\rbr$ satisfy a large deviations principle 
on $C_{0,T}\lb \Delta_d\rb$ 
with the following good rate function 
\be
I_T (\gamma ) = 
\begin{cases}
\int_0^T\cL_0 (\gamma(s),\gamma^\prime (s)) \dd s, \quad &\text{if $\gamma$ is absolutely continuous}\\
\, &\text{and $\gamma (0)=\um$ .}\\
\infty, &\text{otherwise}. 
\end{cases}
\ee
\end{lemma}
By the upper bound in Varadhan's lemma,
\[
 \limsup_{N\to\infty}Z_T^{N, g} (\ump ,\um ) \leq Z_{T}^{g} (\ump ,\um )
\]
On the other hand, by the lower bound in Varadhan's lemma, for each $\delta >0$
\[
 Z_T^g (\ump ,\um ) \leq \liminf_{N\to\infty} \sup_{|\um_1 - \um |<\delta} Z_T^{N ,g} (\ump ,\um_1 ).
\]
Therefore,  \eqref{eq:LD} is a consequence of the claimed continuity of $\um\to Z_T^{N,g} (\ump ,\um )$. 
Let us proceed with establishing the asserted continuity properties of the family  
$Z_T^{N,g} (\cdot ,\cdot )$. 
By reversibility, 
\begin{equation}
 \label{eq:revZTN}
Z_T^{N, g}  (\um , \ump ) =  Z_T^{N ,g}  (\ump , \um ) , 
\end{equation}
so it would be enough to explore those in the second variable only. 

An equivalent task is to check continuity properties of
\be  
Z_T^{N} (\um^\prime , \um )\df 
\frac{1}{N}\log 
\bbE_{\um^\prime}^{N }
{\rm exp}\lbr N \int_0^T  F (\um (t) )\dd t \rbr\1_{\lbr \um (T )=\um\rbr},
\label{eq:ZTN}
\ee
Indeed, by \eqref{eq:rhsg}
\[
Z_T^{N ,g} (\um , \ump ) =  \sqrt{\frac{\mu_N (\um ) }{\mu_N (\ump
)}}Z_T^{N} (\um , \ump )
\]
Now, under $\bbP_N$ the process $\um_N (t)$ is a super-position of $N$ independent particles which hop on 
the finite set  $\cA$ with irreducible rates $\lambda_{\alpha \beta }$.  Since $F$ is bounded, the following
claim is straightforward:  
\begin{lemma}
 \label{lem:Apriori}
There exist  $T_0 > 0$, $\epsilon_0 >0$ and 
a constant $c_1 <\infty$
such that
\begin{equation}
 \label{eq:Apriori}
{\rm e}^{N Z_{T-\epsilon}^N (\ump ,\um )} \geq {\rm e}^{-c_1 N\epsilon }
{\rm e}^{N Z_{T}^N (\ump ,\um )}  .
\end{equation}
uniformly in $N$, $T\geq T_0$, $\epsilon \leq \epsilon_0$, $\ump, \um\in\Delta_d$ and $N$. 
\end{lemma}
Indeed,  trajectories $\um (\cdot )$ on $[0,T]$ and trajectories $\tilde{\um} (\cdot )$ on $[0,T-\varepsilon]$: are related by the following 
one to one map:
$\tilde{\um}(t)=\um(t\frac{T}{T-\varepsilon})$.
Since for some $c_2 = c_2 (T_0 ) >0$, 
up to exponentially small factors, 
the total number of jumps of all the particles is at  most  $c_2N T$, 
the Radon - Nikod\'{y}m derivative is under control and \eqref{eq:Apriori} 
follows. \qed
\smallskip 

As a result,  for any $T, \epsilon$ as above, and for any $\ump , \um^1 , \um^2\in \Delta_d$, 
\begin{equation}
\label{eq:ep-bound1} 
Z_T^N (\ump , \um^2 )\geq   Z_T^N (\ump , \um^1 ) - c_1\epsilon + Z_\epsilon^N (\um^1 , \um^2 ).
\end{equation}
Fix  $\um^1 , \um^2$ and define $\cA_+ = \cA_+ (\um^1 , \um^2 )=\lbr\alpha : m^1_\alpha >
m^2_\alpha\rbr$.  For $\alpha\in \cA_+$ define $\delta_\alpha = m^1_\alpha - m^2_\alpha$. 
One way to drive $\um (\cdot )$ from $\um^1$ to $\um^2$ during $\epsilon$ units of time
is to choose $N\delta_\alpha$ particles out of $Nm^1_\alpha$ for each $\alpha\in\cA_+$, and to 
redistribute them during $\epsilon$ units of time into $\cA\setminus\cA_{+}$, 
without touching the rest of the particles.
  There is an obvious uniform lower bound 
$c_3\epsilon^n$ that a particle starting at the state $\alpha$ will be at state $\beta$ at time 
$\epsilon$.  
We infer:
\begin{equation}
\label{eq:lb-m1m2}
{\rm e}^{N Z_\epsilon^N (\um^1 , \um^2 )}
\geq {\rm e}^{-\lb  \max_\alpha\sum_\beta \lambda_{\alpha ,\beta } -\min F\rb\epsilon N}
\prod_{\alpha\in\cA_+}{Nm^1_\alpha\choose N\delta_\alpha} \lb c_3\epsilon^n\rb^{N\delta_\alpha} .
\end{equation} 
Hence, 
\begin{equation}
 \label{eq:Zep-lb}
Z_\epsilon^N (\um^1 , \um^2 ) \geq -c_4\epsilon - c_5\sum_{\alpha\in \cA_+}\delta_\alpha
\lb d \log \frac{1}{\epsilon} - \log\frac{m^1_\alpha}{\delta_\alpha}\rb .
\end{equation}
Both, the equi-continuity of $\um\to Z_T^N (\ump , \um )$ on $\Delta_d$ and its 
uniform local Lipschitz property on ${\rm int}\lb \Delta_d\rb$ readily follow
from  \eqref{eq:ep-bound1} and \eqref{eq:Zep-lb}.
\qed

\subsection{Proof of Theorem~\ref{thm:constraint}}
\label{sub:PConstraint}
We follow the approach of \cite{F}:  Let $\um\in {\rm int}(\Delta )$ and assume that $u$ is a smooth
function such that $\lbr \um \rbr = {\rm argmin}\lbr u - \psi\rbr$ in a neighbourhood of $\um$.  Then, 
\[
 u(\um )\leq u(\gamma (-t )) +\int_{-t}^0 \cL (\gamma , \gamma^\prime )\dd s -\sfr_1t 
\]
for any $t\geq 0$ and for any smooth curve $\gamma$ with $\gamma (0) =\um$.  Let 
$\uv = \gamma^\prime (0 )$.  Then,  
\[
 \nabla u (\um )\cdot v -\cL (\um , \uv )\leq -\sfr_1 .
\]
Since the above holds for any $\um\in\bbR_0^n$, $\cH (\um , \nabla u (\um ))\leq -\sfr_1$ follows.

In order to check that $\psi$ is a super-solution, note that by the upper and lower bounds on the 
Lagrangian $\cL$ derived in the Appendix, and by the local Lipschitz property of (bounded and continuous) $\psi$   the minimum 
\[
 \min_{\gamma (t_0 )=\um}
\lbr \psi (\gamma (0)) +\int_0^{t_0}\lb\cL (\gamma , \gamma^\prime ) - \sfr_1\rb\dd s\rbr 
\]
is attained at some $\gamma_*$ with $\gamma_*(0)=\ump$ in a $\delta_0$-neighbourhood of $\um$, for all $t_0$ and $\delta_0$ 
appropriately small.  As it is explained in the Appendix, the minimizing curve $\gamma_*$ is 
$C^\infty$ and stays inside ${\rm int}\lb \Delta_d \rb$.  Evidently, 
\[
 \psi (\um ) = \psi \lb \gamma_* (t)\rb +\int_t^{t_0} \lb\cL (\gamma_* , \gamma^\prime_* ) 
- \sfr_1\rb\dd s
\]
for every $t\in [0, t_0 ]$. Assume that $u$ is smooth and ${\rm argmax}\lbr u - \psi\rbr =\lbr \um\rbr$
in a $\delta_0$ neighbourhood of $\um$.  Then, 
\[
 u (\um ) = u (\gamma_* (t_0 ))  \geq u  \lb \gamma_* (t)
\rb +\int_t^{t_0} \lb\cL (\gamma_* , \gamma^\prime_* ) - \sfr_1\rb\dd s
\]
for every $t\in [0, t_0 ]$.  Set $\uv = \gamma_*^\prime (t_0 )$. We infer:
\[
 \nabla u (\um )\cdot \uv\geq \cL (\um , \uv ) - \sfr_1 .
\]
Consequently, $\cH (\um , \nabla u (\um ))\geq  -\sfr_1$.\qed 
\vskip 0.1in

\subsection{Proof of \eqref{eq:C1} of
 Theorem~\ref{thm:A} and Theorem~\ref{thm:ent}}
\label{sub:PMVArPhi}
By Theorem~\ref{thm:LDUniform} and since $\nu_N (\um ) = {\rm e}^{-2N\psi_N (\um )}$ it
follows that 
\[
 \min_{\um } \psi (\um ) = 0. 
\]
Let us rewrite \eqref{eq:C2} as 
\begin{equation}
\label{eq:BoundOnPsi} 
\psi (\um ) = \inf_{\gamma (T) = \um }
\lbr \psi \lb\gamma (0)\rb   +\int_0^T\lb \cL(\gamma , \gamma^\prime ) -\sfr_1\rb
\dd t \rbr 
\end{equation}
Since $V (\um ) = \cL (\um , 0 )\leq \cL (\um , \uv )$, the above 
might be possible only if 
$\sfr_1 =  \min_{\um} V (\um )$. 
\smallskip 

Furthermore, the Lagrangian $\cL$ is uniformly super-linear in the second variable: 
 By \eqref{eq:lbL0} of the Appendix for every $C>0$ and $\delta >0$ we can find $T>0$ such that
\[
 \inf_{{\rm diam}( \gamma ) >\delta}\int_0^T \lb \cL(\gamma , \gamma^\prime ) -\sfr_1\rb \dd t \geq C. 
\]
Which means that for $C > \max \psi$, the contribution to \eqref{eq:BoundOnPsi} for $\gamma$-s with the 
diameter larger than $\delta$ could be ignored. 

Let, therefore,   ${\rm diam}(\gamma ) \leq \delta $ .
By \eqref{eq:C11},
\begin{equation}
\label{eq:ConvL} 
\int_0^T\lb \cL(\gamma , \gamma^\prime ) -\sfr_1\rb
\dd t \geq T\min_{\ump\in\gamma }\lb V (\ump ) - \sfr_1\rb 
\end{equation}
We infer:
\begin{equation}
 \label{eq:MinPsi}
\psi (\um ) \geq \min_{|\ump - \um |\leq\delta }\psi (\ump) + T \min_{|\ump - \um |\leq\delta}\lb 
V (\ump )  -\sfr_1\rb .
\end{equation}
The claim of Theorem~\ref{thm:ent}  follows
as soon as we notice that all the minima of $\um\mapsto V (\um  )$ belong to ${\rm int}\lb\Delta_d\rb$. 
\qed

\section{
Results for Spin-$\frac12$ Model
}
\label{sec:S12}
For spin-$\sfs$ models \eqref{eq:spins}
\be 
\label{eq:r1S}
\sfr_1 = \min_{\um} V(\um )=  \min_{\um}\lbr \frac{\lambda}{4} \sum_{|\alpha -\beta |=1 }
\sqrt{\sfs (\sfs +1 ) - \alpha\beta} \lb \sqrt{m_\alpha} - \sqrt{m_\beta}\rb^2 - 
F\lb
\um 
\rb  \rbr .
\ee

In spin-$\frac12$ case it is 
convenient to take $\{-1,1\}$ instead of  $\{-\frac12,\frac12\}$ as a set of classical labels for Spin-$\frac12$ Model. 
The  Hamiltonian is given by 
\[
-\cH_N=N F\lb 
M_{-1}^N , M^N_1 \rb +\lambda\sum_i\hat{\sigma}^x_i
\] 
where
\begin{equation}
\hat{\sigma}^x\, =\, \lb
\begin{array}{cc}
0  &1\\
1 &0
\end{array}
\rb\qquad
\text{and}\qquad
\hat{\sigma}^z\, =\, \lb
\begin{array}{cc}
1  &0\\
0 &-1
\end{array}
\rb .
\end{equation}
In this notation $\hat{\sigma}^z\ket{\alpha}=\alpha\ket{\alpha}$ and $\hat{\sigma}^x\ket{\alpha}= \ket{-\alpha}$ 
for $\alpha =\pm 1$. 

The simplex $\Delta_d$ is just a segment $\lbr (m_{-1} , m_{1}): m_{-1} +m_{1} =1\rbr$, parametrized by
a single variable $m = m_1 - m_{-1}\in [-1, 1]$. Any vector $\utheta\in \bbR^2_0$ is of 
the form $\utheta = (\theta , -\theta )$. 
Define $F (m ) = F\lb \frac{1-m}{2} , \frac{1+m}{2}\rb$. 
Thus, in terms of $m$ and $\theta$, the Hamiltonian in \eqref{eq:HL} is 
\begin{equation}
\label{eq:HCW} 
\cH (m , \theta ) = 
\lambda\sqrt{1 - m^2}\, {\rm cosh}( 2\theta ) - \lambda + F (m) .
\end{equation}
Consequently, the effective potential 
\be 
\label{eq:effV}
V (m ) = -\cH (m ,0) = 
\lambda - \lb \lambda\sqrt{1 - m^2} + F (m)\rb ,
\ee
and the asymptotic leading eigenvalue $\sfr_1$
is  given by
\be
\label{eq:V12}
\sfr_1 = \min_{m\in (-1,1 )} V(m ) = \lambda - \max_{m\in [-1,1 ]}\lbr \lambda\sqrt{1 - m^2} + F (m)\rbr .
\ee
\subsection{Minima of $V$}
\label{sub:minV}
In order to explore the minimization problem \eqref{eq:V12} it 
would be convenient to 
represent $ F (m  ) = - G \lb {\rm sign} (m) \sqrt{1-m^2 }\rb$. Then, 
\begin{equation}
 \label{eq:r1OneHalf}
\sfr_1   = \lambda - \max_{-1 \leq t\leq 1}\lbr \lambda |t | - G ( t)\rbr .
\end{equation}
Indeed, as it clearly seen from \eqref{eq:V12} (and as it follows in general  by Theorem~\ref{thm:ent}),
 all the minima of $V $ belong to 
$(-1,1)$, a possible jump discontinuity of $G$ at zero (if $ F (-1 )\neq  F (1)$) plays no role 
for the computation of 
maxima.  Note also that $G (-1) = G (1) = F (0 )$. 

 Let
\[
 \cM_\lambda = {\rm argmim}\lb V\rb = \lbr m\in (-1,1) : \cH (m , 0) = -\sfr_1\rbr .
\]
In other words,
 let $\cT_\lambda  \subset [-1,1]$ be the set of maximizers in \eqref{eq:r1OneHalf}. 
We set $\cT_\lambda = \cT_\lambda^+\cup\cT_\lambda^-$, where $\cT_\lambda^+ = 
\cT_\lambda\cap (0,1)$. 
Then, 
\begin{equation} 
\label{eq:ATlambda}
m\in\cM_\lambda \Leftrightarrow t= {\rm sign} (m) \sqrt{1-m^2}\in \cT_\lambda^{{\rm sign} (m)}  .
\end{equation}
By assumption {\bf A1}, $G$ is a polynomial of finite degree on each of the intervals 
$[-1,0)$ and $(0, 1]$ , so the set $\cT_\lambda$ is 
finite, and its maximal cardinality is ${\rm deg}(G )-1$.  Various options are depicted on 
Figure~\ref{fig:G} (for simplicity we depict only the $(0,1]$ interval and, accordingly, the set 
$\cT_\lambda^+$): 

(a)  First of all there exists $\lambda_c\in [0,\infty)$, such that $\cT_\lambda =\lbr \pm 1\rbr $ 
on $(\lambda_c ,\infty )$. 

(b) If $\lambda_c >0$, then $\cT_{\lambda_c}$ still contains $\pm 1$. It could happen, however,
that $\cT_{\lambda_c} = \lbr -1 ,t_1, \dots , t_k ,1\rbr$ contains at most $k\leq {\rm deg}(G )/2$ other points.
 In the latter case  $(-1,0)\cup (0 ,1 )$ necessarily contains at least $2k$  inflection 
points  of $G$. 

(c) There might be other exceptional values of $\lambda <\lambda_c$ for which either of 
$\cT_\lambda^\pm$ 
 is not a singleton. 
If, for instance $\cT_\lambda^+ =\lbr t_1^\lambda ,\dots ,  t_{k }^\lambda \rbr  $ is not a singleton, then 
the interval $( t_1^\lambda , t_{k }^\lambda )$ contains 
at least $2(k-1)$ inflection points of $G$. Since there are at most ${\rm deg}(G )-2$ inflection 
points all together, and since intervals spanned by  different $\cT_\lambda^\pm$ are disjoint, we infer
that $\cT_\lambda^\pm$ is not a singleton for at most ${\rm deg}(G )/2$ values of $\lambda$.  
\smallskip 

Values of $\lambda$ for which the cardinality of $\cT_\lambda$ changes  
correspond to first order 
phase transitions in the ground state.  
\smallskip 

\begin{figure}
\begin{center}
\includegraphics[width=0.30\textwidth]{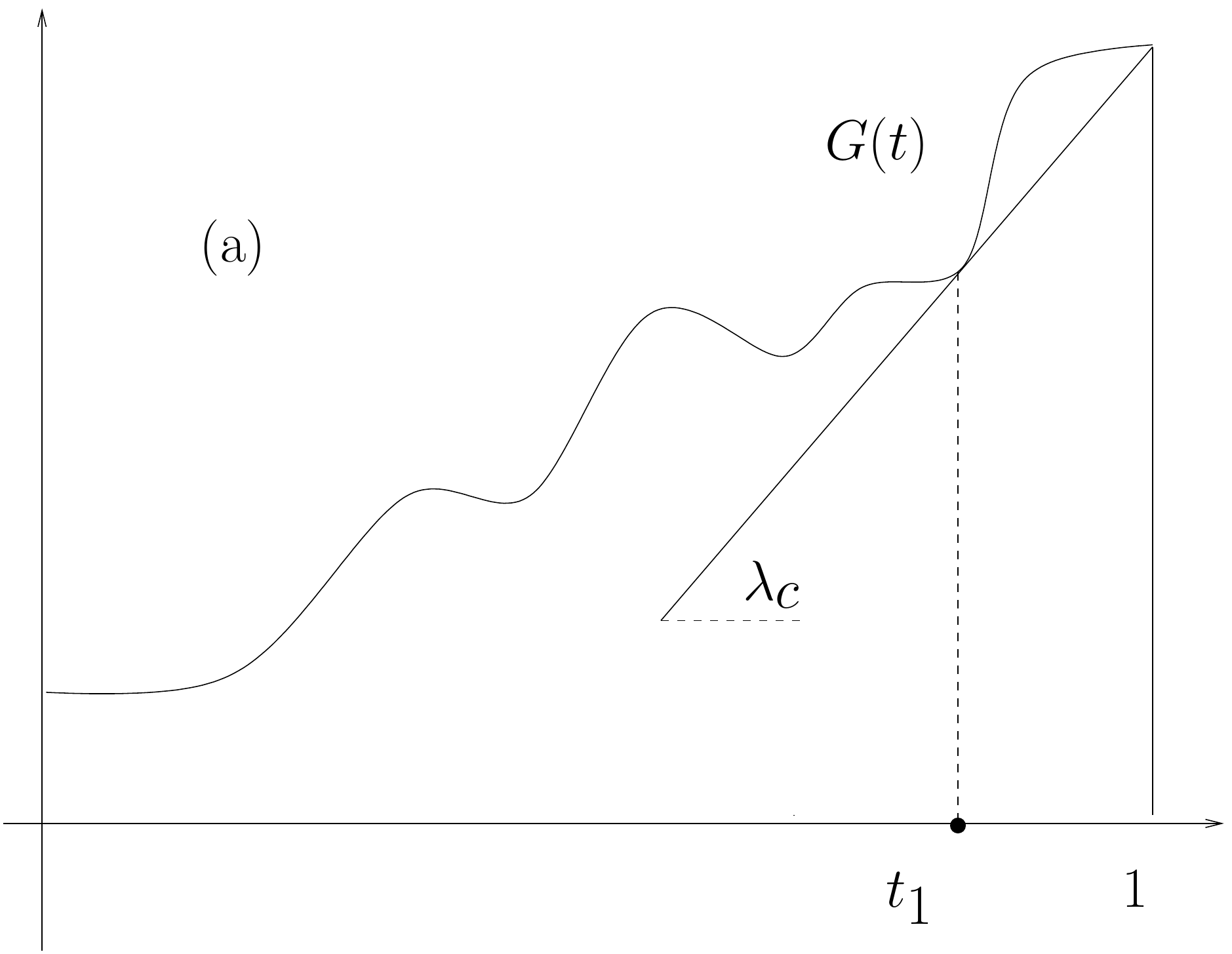}
\includegraphics[width=0.30\textwidth]{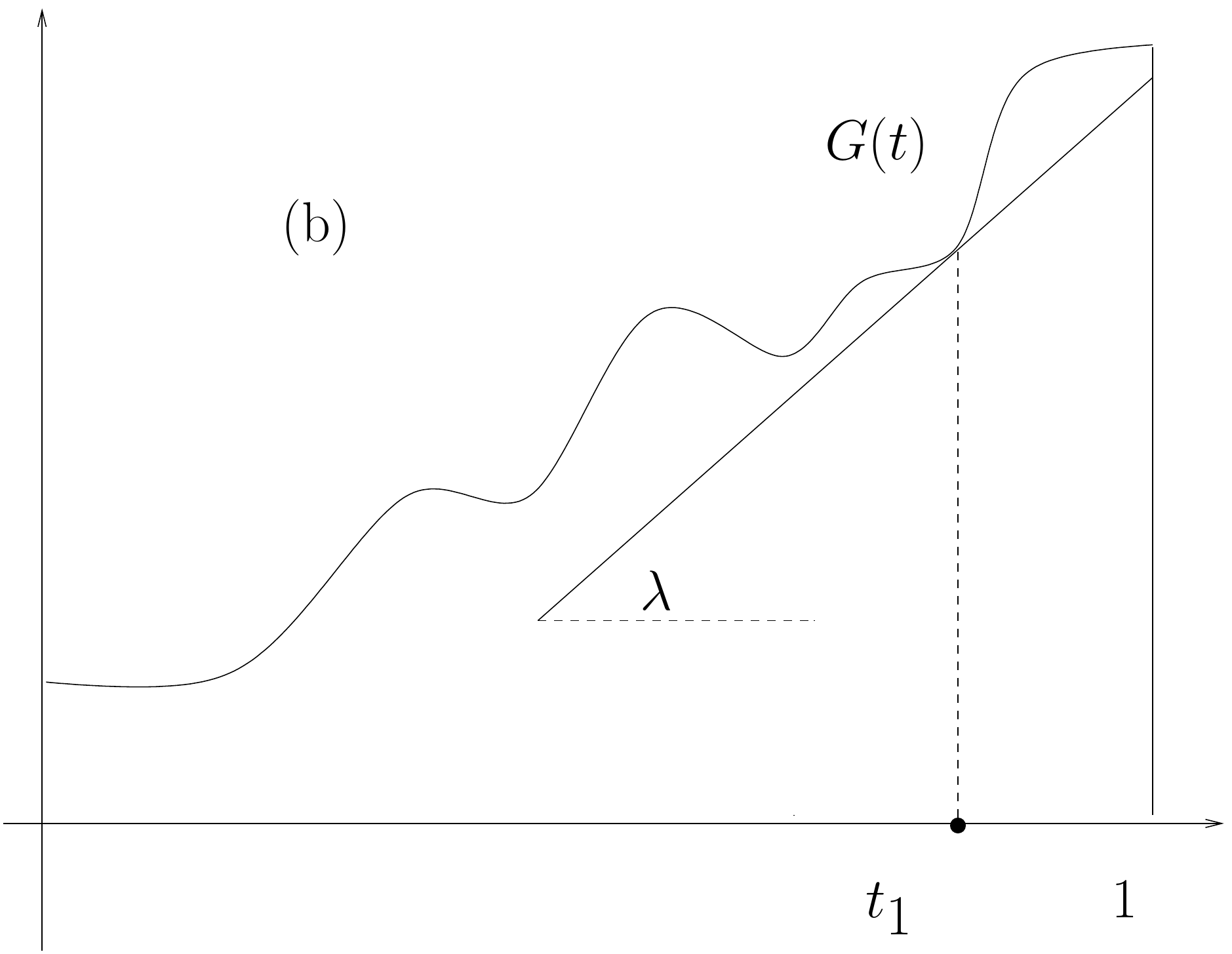}
\includegraphics[width=0.30\textwidth]{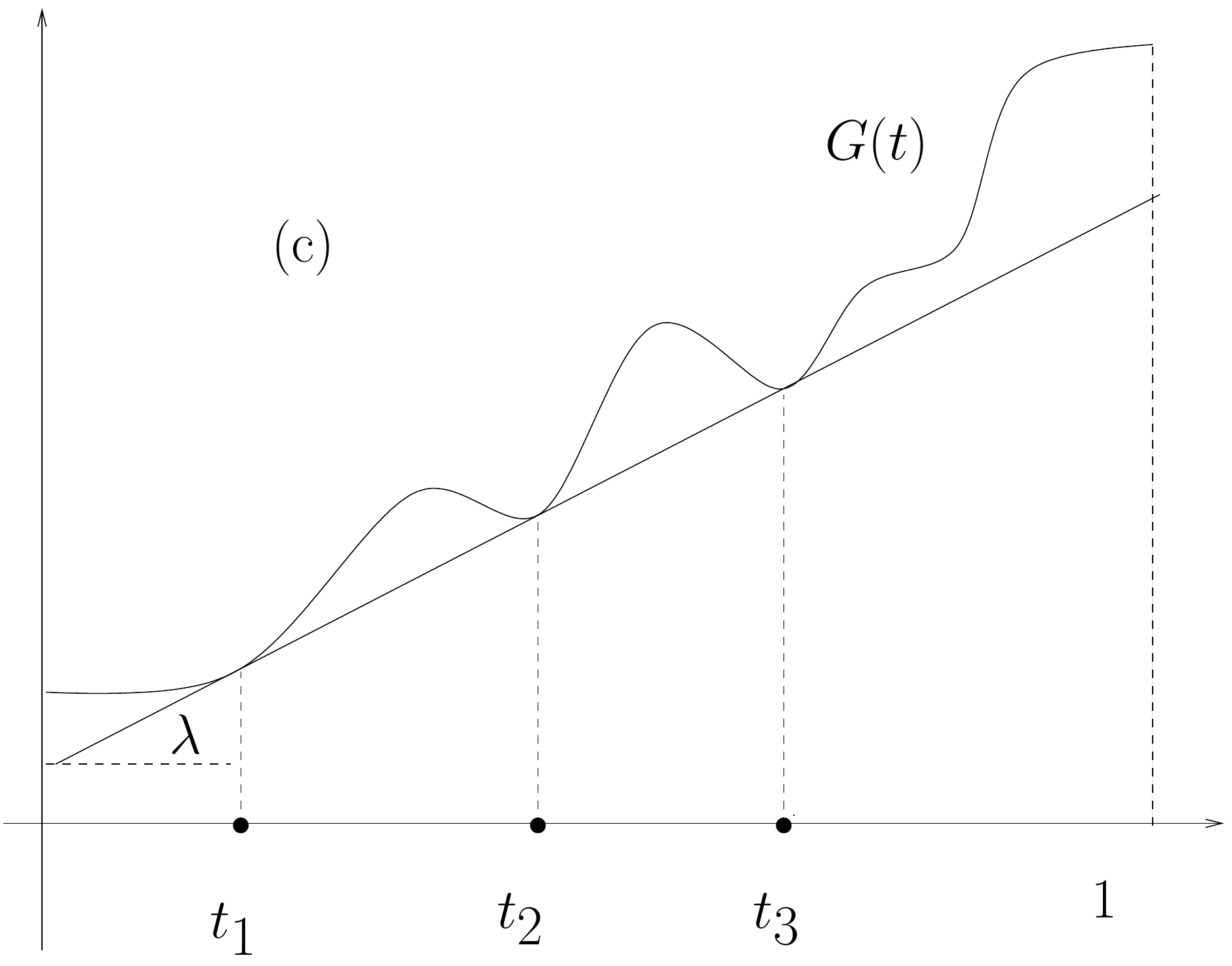}
\end{center}
\caption{ (a) The critical $\lambda_c$ and $\cT_{\lambda_c}^+ = \lbr t_1 , 1\rbr$. 
 (b) $\cT_\lambda^+$ is a singleton. \newline 
(c) $\cT_\lambda^+ =\lbr t_1, t_2 ,t_3 \rbr$. There are at least 
two inflection points of $G$  on $(t_1 ,t_3)$}  
\label{fig:G}
\end{figure}
\subsection{Ferromagnetic $p$-body interaction}
\label{sub:p}
In the usual  Curie-Weiss case with pair interactions 
$G(t ) = \frac12\lb t^2 -1\rb$, so that $\sfr_1 = -\frac{(\lambda-1)^2}{2} $ if $\lambda\leq 1$,
and, accordingly, $\sfr_1  = 0$ if $\lambda\geq 1$. 
For $\lambda \geq 1$ the set $\cM_\lambda = \lbr 0 \rbr$. For $\lambda \in (0,1)$, 
$\cM_\lambda = \lbr \pm\sqrt{1-\lambda^2}\rbr$. 
No first order transition occurs. 
\smallskip 

In  the $p>2$-body ferromagnetic interaction case \eqref{eq:p-body} the function 
\[
G(t) = -{\rm sign}(t)^p \lb 1-t^2\rb^{p/2} .
\]
For odd $p$ maximizers of $\lambda|t| - G (t)$ 
always lie in $(0, 1]$. For even $p$ the set $\cM_\lambda$ is symmetric.  Thus in either case
it is enough to consider 
\begin{equation}
 \label{eq:max-pbody}
\sfr_1 = \lambda - \max_{t\in (0,1]} \lbr \lambda t + \lb 1-t^2\rb^{p/2}\rbr .
\end{equation}
The crucial difference between the Curie-Weiss case $p=2$ and $p>2$ is that 
in the 
latter situation, $G^\prime (1) = 0$, and $G$ contains an inflection point $t_p = \sqrt{\frac{1}{p-1}}$
inside $(0,1 )$. 
An easy computation reveals that for $p>2$, 
\begin{equation}
\label{eq:p-critical}
\lambda_c = \frac{p}{p-1}\lb 1 - \frac{1}{(p-1)^2}\rb^{\frac{p}{2} - 1} \quad{\rm and}\quad
\cT_{\lambda_c}^{+} = \lbr\frac{1}{p-1}  , 1\rbr .
\end{equation}
Accordingly,  for even $p$, 
\begin{equation}
 \label{eq:Aset-p}
\cM_{\lambda_c} = \lbr 0, \pm \hat m \rbr = \lbr 0, \pm  \sqrt{\frac{p(p-2)}{(p-1)^2}}\rbr ,
\end{equation}
whereas  for odd $p$;  $\cM_{\lambda_c} = \lbr 0 ,\hat m\rbr = 
 \lbr 0,   \sqrt{\frac{p(p-2)}{(p-1)^2}}\rbr$.  This is precisely
formula (14) of \cite{BS12}.  For $\lambda > \lambda_c$ 
 the set $\cM_\lambda = \lbr 0\rbr$.
For $\lambda <\lambda_c$ there exists $m^* = m^* (\lambda ,p )\in \lb 
\sqrt{\frac{p(p-2)}{(p-1)^2}} , 1\rb$ such that 
 the set $\cM_\lambda$ is a singleton $\lbr m^* \rbr$ in the odd case, whereas 
$\cM_\lambda = \lbr \pm m^* \rbr$ in the even case. Thus, for mean-filed models 
with $p$-body interaction, $\lambda_c$ is the only value at which fist order transition
 in the ground state  occurs.  

\subsection{Asymptotic ground states} 
\label{sub:gs}
Let us return to general polynomial interactions $F$. 
Fix  $m\in (-1,1)$ and consider the equation, 
\begin{equation}
 \label{eq:HJtheta}
\cH (m , \theta ) = -\sfr_1 .
\end{equation}
The Hamiltonian $\cH (m ,\cdot )$ is strictly convex and symmetric. Hence, if 
$m\in\cM_\lambda$, then $\theta = 0$ is
the unique solution. If $m\not\in\cM_\lambda$, then necessarily $\cH (m , 0) < -\sfr_1 $. Hence,
there exist $\theta (m) >0 $, such that $\pm \theta  (m)$ are the unique
solutions to \eqref{eq:HJtheta}.  If $F$ is symmetric, then 
 $\theta (-m ) = -\theta (m)$.  In any case, however, the following holds:

Let $\psi$ be an 
admissible  solution of $\cH (m , \psi^\prime  ) = -\sfr_1$. 
Since $\psi$ is locally 
Lipschitz,  it is a.e. differentiable. Consequently, $\psi^\prime  (m) =\pm \theta  (m)$ a.e. on $(-1,1)$.  
The proposition below relies only on the fact that $\psi$ is a viscosity solution on $(-1,1 )$. 
\begin{proposition}
 \label{prop:ViscosityR} Record $\cM_\lambda = \lbr -1 <   m_1, \dots ,m_k <1\rbr$ in the 
increasing order.  Set $m_0 = -1$ and $m_{k+1} = 1$.  Then on each of the intervals  
$[m_\ell , m_{\ell +1}]$ 
the gradient $\psi^\prime$ is of the following form: There exists $m^*_\ell \in [m_\ell , m_{\ell +1}] $, 
such that:
\begin{equation}
\label{eq:mstarl} 
\psi^\prime = \theta\ {\rm on}\  [m_\ell , m_{\ell}^*)\quad {\rm and} \quad 
\psi^\prime = - \theta\ {\rm on}\  (m_\ell^* , m_{\ell +1}] .
\end{equation}
\end{proposition}
\noindent
{\em Proof.} 
It would be enough to prove the following: If $m\in (m_{\ell}, m_{\ell+1})$ and $\psi^\prime (m) = 
\theta  (m)$, then for any $n\in ( m_\ell ,m)$,
\begin{equation}
 \label{eq:interval}
\psi (n) = \psi (m) - \int_n^m\theta  (t )\dd t .
\end{equation}
Recall that since $\psi$ is  a viscosity solution on $(-1,1)$, then 
\begin{equation}
\label{eq:supercond} 
\liminf_{\varepsilon\to 0} \frac{\psi(n^*+\varepsilon)-\psi(n^*)}{ |\varepsilon|}
 \geq\theta  \Rightarrow \theta\notin(- \theta (n^*),\theta (n^*)).
\end{equation}
for any $n^*\in (-1,1 )$. We shall show that if \eqref{eq:interval} is violated for 
some $n\in (m_\ell , m )$, then \eqref{eq:supercond} is violated as well in the sense that
there exists $n^* \in (n,m )$ and $\theta\in ( -\theta(n^*),\theta (n^*))$ such that the right hand side of 
\eqref{eq:supercond} holds. 

Indeed, since $\psi^\prime (t) = \pm\theta (t)$ a.e. on $(-1,1)$ it always holds that
\[
 \psi (n) \geq \psi (m) - \int_n^m\theta  (t )\dd t
\]
Let us assume strict inequality. For $k\in [n,m]$ define
\[
 v (k ) = \int_n^k \frac{\psi^\prime (t) +\theta (t )}{2\theta (t )}\dd t .
\]
By construction $v(n)= 0$, $v (m ) \df p (m-n ) <  (m-n )$ and $v^\prime (m )=1$. 
There is no loss of generality to assume that $p>0$. 
Hence there exists
$n^*\in (n, m)$ such that $p\in\partial v (n^* )$ (see Figure~\ref{fig:prop41}). By continuity
of $\theta (m )$ this would mean that
\[
 \liminf_{\varepsilon\to 0} \frac{\psi(n^*+\varepsilon)-\psi(n^*)}{ |\varepsilon|}
 \geq p\theta  (n^* ) - (1-p )\theta  (n^*) = (2p -1)\theta (n^* ) . 
\]
Since for any inner point $n^*\in (m_\ell , m_{\ell +1 })$; $ \theta (n^* ) > 0$,
we arrived to a contradiction. \qed
\begin{figure}
\begin{center}
\includegraphics[width=0.8\textwidth]{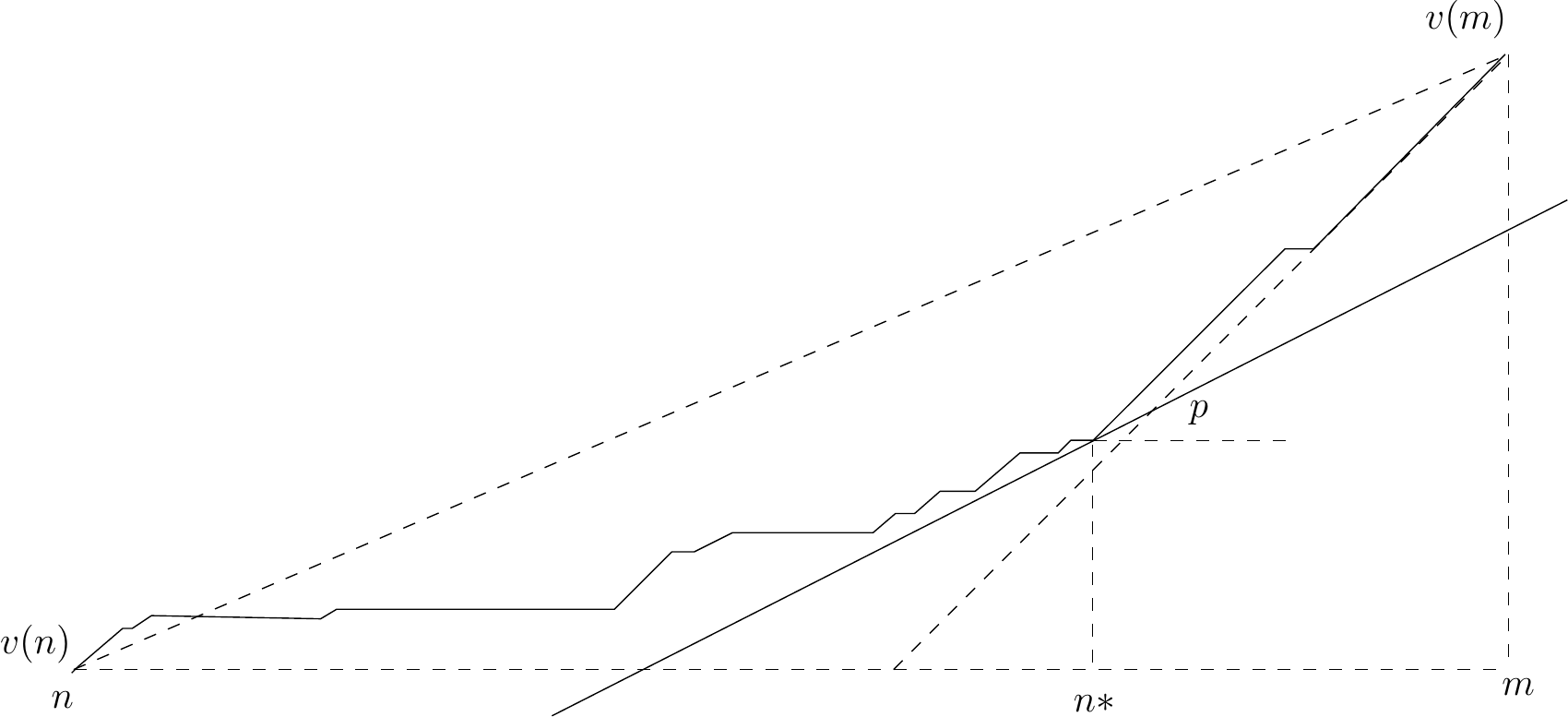}
\end{center}
\caption{ $p\in\partial v (n^* )$.}  
\label{fig:prop41}
\end{figure}
\smallskip  

\begin{remark}
Note that Proposition~\ref{prop:ViscosityR} implies that ground states $\psi$ with more than one 
local minimum necessarily develop shocks.
\end{remark}
If $m_\ell < m_\ell^* <m_{\ell +1}$ from Proposition~\ref{prop:ViscosityR}, then $m_\ell^*$ is a local maximum of $\psi$, 
and $m_\ell^*$ is a   shock  location for the stationary Hamilton-Jacobi equation 
$\cH (m ,\psi )=- \sfr_1$. 

If $\psi$ is, in addition,  a weak KAM solution (in particular, if $\psi$ is admissible), then 
 ${\rm argmin}\lbr \psi \rbr \subseteq \cM_\lambda $. Consequently, by Theorem~\ref{thm:ent},  in the latter
case, $\psi^\prime  = -\theta$ on $(-1, m_1 )$ and $\psi^\prime  = \theta$ on $(m_k ,1)$. 

 Admissible solutions are always normalized in the sense that $\min\psi = 0$ and the 
minimum is attained on $\cM_\lambda\subset (-1,1 )$. It follows that
admissible are uniquely defined in the following two cases:

\case{1}.  The set $\cM_\lambda = \lbr m^*\rbr$ is a singleton. Then,  $\psi^\prime = -\theta$  on 
$ (-1, m^* )$ and $\psi^\prime = \theta $ on $(m^* , 1)$. Consequently,
\be 
\label{eq:psi12}
\psi (m) = \left| \int_{m^*}^{m } \theta (t)\dd t\right| .
\ee

\case{2}. The interaction $F$ is symmetric and $\cM_\lambda = \lbr \pm m^*\rbr$. Then, $\psi$
is also symmetric; $\psi^\prime = \theta$ on $(-m^* ,0)\cup (m^* ,1)$ and $\psi^\prime = -\theta$
on $(-1,0)\cup (0,m^* )$.  That is $\psi (m ) = \psi (-m )$ and $\psi$ is still given by 
\eqref{eq:psi12} for $m\geq 0$. Note that in this case $\psi^\prime$ has a jump at $m =0$. 
\medskip 

Let $\Lambda^c$ be the set of $\lambda$ which do not fall into one of the two cases above. 
As we have seen in Subsection~\ref{sub:p},  $\Lambda_c = \emptyset$ in the case of Curie-Weiss
model, and $\Lambda_c = \lbr\lambda_c\rbr$ (see \eqref{eq:p-critical}) for general $p$-body interaction. 

\subsection{Multiple Wells}
\label{sub:wells}
We shall refer to $\lambda\in\Lambda_c$ as to the case of multiple wells. 
Note first of all that there is a continuum of normalized 
solution of \eqref{eq:C2} as soon as the cardinality $|\cM_\lambda |\geq 2$.  
Indeed, it is easy to see that any normalized $\psi$ which complies with the conclusion of 
Proposition~\ref{prop:ViscosityR}  will be a solution to \eqref{eq:C2}. 

 One needs, therefore,  an additional  criterion to determine locations of shocks $
\lbr m_\ell^*\rbr $ or, equivalently, to determine   values $\lbr \psi (m_\ell )\rbr$ 
 for    admissible solutions. It would be tempting to derive location of shocks by some 
natural limiting procedure via stabilization of shock propagation along 
 Rankine-Hugoniot curves. Since however, we arrived to 
\eqref{eq:C2} directly from the eigenvalue equation without recourse to a finite horizon problem, it
was not clear to us which limit to consider. 
Our  selection of admissible solutions to  \eqref{eq:C2} is 
 based on a refined asymptotic analysis of Dirichlet eigenvalues in a vicinity of 
points belonging to the set $\cM_\lambda$ .  Namely, a  point $m_\ell \in \cM_\lambda$ can be local 
minima of an admissible solution 
$\psi$ only if there is an exponential splitting of the corresponding bottom eigenvalues.  Precise result is 
formulated in Proposition~\ref{prop:minpsi} below. 

The results of \cite{KR1, KR2, KR2}. enable to explore asymptotic expansions of such eigenvalues
with any degree of precision.  In the simplest case we deduce the following: 
\begin{corollary}
 \label{cor:curv}
Assume that
\be 
\label{eq:curv}
\min_{m\in \cM_\lambda} \chi_0 (m) \df 
\min_{m\in \cM_\lambda} 
\lbr \frac{\lambda}{1-m^2} - \sqrt{1-m^2} F^{\prime\prime}  (m) \rbr 
\ee
is attained at either a  unique point $m^*$ (non-symmetric potentials) or at a unique
 couple $\pm m^*$ (symmetric potentials). Then there is a unique admissible solution 
$\psi$, which is still given by \eqref{eq:psi12}. 
\end{corollary}
\noindent
For instance,  in the critical  ($\lambda =\lambda_c$)  case of $p >2$ body interaction, 
a substitution of \eqref{eq:p-critical} and \eqref{eq:Aset-p} 
yields:
\be 
\label{eq:pspin-zero}
\chi_0 (0 ) = \lambda_c\quad{\rm and}\quad \chi_0 (\hat m ) = (p-2)(p-1 )\lambda_c .
\ee
In other words, $\chi_0 (0) <\chi_0 (\hat m )$, for any $p>2$ and $\lambda = \lambda_c (p)$  
Consequently,  even at $\lambda = \lambda_c$ 
there is still a unique admissible solution 
$\psi (m ) = \abs{\int_0^m \theta (t )\dd t}$ with the unique minimum at $m^*=0$. 

\noindent
We explain Corollary~\ref{cor:curv} in the concluding paragraph of this Section. 
\medskip

\noindent
\paragraph{\bf Spectral Asymptotics and the Set $\cM_\lambda$} 
Assume that $\lambda\in\Lambda_c$ and, as before, denote $\cM = \lbr m_1, \dots m_k\rbr$. 
\begin{lemma}
 \label{lem:expdecay}
For any $\delta >0$ there exists $\epsilon >0$ such that
\be  
\label{eq:decaypsi}
\min_{\dd (m , \cM_\lambda )\geq \delta} \psi (m) \geq \epsilon ,
\ee
uniformly in normalized admissible solutions of \eqref{eq:C2}.
\end{lemma}
 \begin{proof}
Let $m\in ( m_l ,m_{l+1})$. By Proposition~\ref{prop:ViscosityR} 
\[
 \psi (m) \geq\min\lbr \psi (m_l ) +\int_{m_l}^m \theta( t)\dd t\, ,\, 
\psi (m_{l+1} ) +\int_{m }^{m_{l+1}} \theta( t)\dd t \rbr ,
\]
and \eqref{eq:decaypsi} follows.
 \end{proof}
In the sequel $h_N = {\rm e}^{-N\psi_N}$ is the Perron-Frobenius 
eigenfunction of $\cG_N^g + NF_g \df \cS_N $ ; $\cS_N h_N = -R_N^1 h_N$. 
Recall:
\be  
\label{eq:operator12}
\begin{split}
\cS_N f (m) &=  N \lb F (m ) - \lambda \rb   f (m) \\
&+ \frac{N\lambda}{2}
\sqrt{(1-m)(1+m +\frac{2}{N})} f\lb m+\frac{2}{N}\rb \\
&
+
 \frac{N\lambda}{2}
\sqrt{(1+m)(1-m +\frac{2}{N})} f\lb m-\frac{2}{N}\rb 
\end{split}
\ee 
Pick $0 < \delta <\frac{1}{4}\min_{l}\abs {m_{l+1} -m_l}$. Let $1\equiv \sum_0^k\chi_l$ be a partition of unity
satisfying: For $l=1, \dots , k$
\[
 \chi_l \equiv 1\ {\rm on}\ I_\delta (m_l ) \quad{\rm and}\quad 
\chi_l \equiv 0\ {\rm on}\ I_{2\delta}^c (m_l ) .
\]
Above $I_\eta (m )$ is the interval  $ [m-\eta , m+\eta]$.  By Lemma~\ref{lem:expdecay} there
exists $\epsilon >0$ such that for $l =1, \dots ,k$ and all $N$ large enough
\be
\label{eq:partition}
 \frac{1}{N}\log \max_{m}\abs{\lb \cS_N +R_N^1\rb \chi_l  h_N (m)}\leq -\epsilon .
\ee
Let $\cS_N^l$ be a Dirichlet restriction of $\cS$ to $I_\delta (m_ l )$. 
Let $- R_{N ,l}^1$  be the leading eigenvalue of $\cS_{N}^l$. 

\noindent
We are entitled to conclude: There exists $\epsilon^\prime >0$ such that
\be 
\label{eq:Dir}
 \frac{1}{N}\log \abs{ R_N^1 - \min_l R_{N,l}^1}\leq -\epsilon^\prime .
\ee
Furthermore, 
\begin{proposition}
 \label{prop:minpsi}
 If 
$l=0, \dots , k$ and 
$\psi = \lim_{j\to\infty} \psi_{N_j}$
 is a subsequential limit such that $m_l$ is a local minimum of $\psi$, then 
there exists $\epsilon^\prime  >0 $ such that: 
\be 
\label{eq:expsplit}
\frac{1}{N_j}  \log\abs{ R_{N_j }^1 -  R_{N_j , l}^1} \leq -\epsilon^\prime .
\ee
\end{proposition}
\begin{proof}
In view of Lemma~\ref{lem:GS} the claim readily follows from the general theory of 
exponentially  low 
lying spectra for  metastable Markov chains \cite{BEGK}. For a direct proof note that  
 under the assumptions of the Proposition, one (possibly after further  shrinking the value of $\delta$) 
can upgrade 
\eqref{eq:partition} as 
\be 
\label{eq:expsplit-l}
\frac{1}{N_j }\log \max_{m}\abs{\lb \cS_{N_j} +R_{N_j}^1\rb \frac{ \chi_l  h_{N_j} (m)}{h_{N_j}
 (m_l )}}\leq -\epsilon ,
\ee
and \eqref{eq:expsplit} follows from the spectral theorem. 
\end{proof}
\smallskip 

\noindent
\paragraph{\bf Asymptotics of Dirichlet Eigenvalues $R_{N ,l}^1$}
Define $\lambda (m) = \sqrt{1-m^2}$. The asymptotics of $R_{N,l }^1$ up to zero order terms
is given \cite{KR3} by 
\be
\label{eq:eqRN-zero}
-R_{N ,l}^1 = - N\sfr_1 -\sqrt{\frac{V^{\prime\prime} (m_l )}{\lambda (m_l )}} + {\sf O}\lb\frac{1}{N}\rb
= - N\sfr_1 -\chi_0 (m_l ) + {\sf O}\lb\frac{1}{N}\rb , 
\ee
where we used the explicit expression \eqref{eq:effV} for $V$ in 
the second equality. $\chi_0$ was defined in \eqref{eq:curv}. 
The claim of Corollary~\ref{cor:curv}  follows now from Proposition~\ref{prop:minpsi}. 
\qed
\smallskip

\appendix
\section{The variational problem}
The Lagrangian $\cL_0$ was defined in \eqref{eq:HLnot}

\vskip 0.1in
\paragraph{\bf Lower bounds on $\cL_0$.}
Fix $\alpha \in \cA$ and consider $\theta^t_\alpha = \frac{n-1}{n}t $ and, for $\beta\neq \alpha$, 
$\theta_\beta^t = -\frac{1}{n}t $. Recall that $\uv\in\bbR^n_0$, that is $v_\alpha = -\sum_{\beta\neq \alpha}
v_\beta$. Therefore, for any $\alpha$
\[
 \cL_0 (\um , \uv ) \geq \sup_{t}\lbr tv_\alpha - \sum_{\beta\neq\alpha}
\sqrt{m_\alpha m_\beta}\, 
\lambda_{\alpha \beta} 
\lb 
{\rm cosh} ( t ) -1 \rb\rbr 
\]
Define $\lambda_\alpha (\um ) = \sum_{\beta\neq\alpha}\sqrt{m_\alpha m_\beta}\, 
\lambda_{\alpha \beta} $. For $\abs{v_\alpha }\geq \lambda_\alpha  (\um )$ 
one may choose 
$t^* = {\rm sign}(v_\alpha )\log\frac{\abs{v_\alpha}}{\lambda_\alpha (\um )}$. 
We infer: If $\abs{v_\alpha }\geq \lambda_\alpha  (\um )$, then
\begin{equation}
 \label{eq:lbL0}
\cL_0 (\um , \uv ) \geq \abs{v_\alpha}\lb \log\frac{\abs{v_\alpha}}{\lambda_\alpha (\um )} -1\rb . 
\end{equation}

\vskip 0.1in
\paragraph{\bf Upper bounds on the Lagrangian $\cL_0$.}
Consider
\[
 \cR_0 (\um , \uv ) \df \sup_{\utheta}\lbr \sum v_\alpha \theta_\alpha - 
\sum_{\alpha ,\beta} \sqrt{m_\alpha m_\beta}\, \lambda_{\alpha ,\beta }\,
 {\rm cosh} (\theta_\beta -\theta_\alpha )\rbr
.\]
Since  $\cL_0 (\um ,\uv ) = \cR_0 (\um , \uv ) + \sum_\alpha \lambda_\alpha (\um )$, 
it would be enough to control the dependence of $\cR_0$ on $\uv$. 
\smallskip 

\noindent
Let us say that a flow $\uf = \lbr f_{\alpha  \beta }\rbr$ is compatible with $\uv\in\bbR_0^n$; 
$\uf\sim\uv$ if:
\begin{itemize}
 \item[(a)] It is a flow: $f_{\alpha  \beta } = - f_{\beta \alpha }$.
\item[(b)] For any $\alpha\in\cA$, $\sum_\beta f_{\beta\alpha } = v_\alpha$. 
\end{itemize}
 Then $\sum v_\alpha\theta_\alpha = \frac12\sum_{\alpha ,\beta} (\theta_\beta - \theta_\alpha )f_{\alpha\beta}$.
Hence, for any $\uf\sim \uv$, 
\begin{equation}
 \label{eq:flow}
\cR_0 = \sup_{\utheta} \lbr \frac12\sum_{\alpha ,\beta} (\theta_\beta - \theta_\alpha )f_{\alpha\beta}
- \sqrt{m_\alpha m_\beta} \, \lambda_{\alpha \beta } \, 
{\rm cosh} (\theta_\beta -\theta_\alpha )\rbr .
\end{equation}
We shall rely on the following upper  bound on each term in \eqref{eq:flow}: For any $f$ and $a >0$
\[
\sup_t \lbr ft - 
a\, {\rm cosh}(t)\rbr \leq \abs{f}\log\lb 1+\frac{2\abs{f}}{a}\rb .
\]
Consequently, we derive the following upper bound on $\cR_0$:
\begin{equation}
 \label{eq:UpperFlow}
\cR_{0} (\um ,\uv)\leq \inf_{\uf\sim\uv}\, 
\sum_{\alpha ,\beta} \frac{\abs{f_{\alpha\beta}}}{2}\log\lb 1+  
\frac{\abs{f_{\alpha\beta}} }{\sqrt{m_\alpha m_\beta} \, \lambda_{\alpha\beta}}\rb  .
\end{equation}
\smallskip 

\noindent
\paragraph{\bf Regularity of minimizers.}  Let $\um\in {\rm int}\lb \Delta_d\rb$. We claim that there exists $\delta_0 >0$
and $t_0 >0$ such that for any $\ump$ in the $\delta_0$-neighbourhood of $\um$  the minimizer $\gamma^*$
of 
\[
 \inf_{\gamma (0 )=\ump, \gamma (t_0 )=\um}\int_0^{t_0}\cL (\gamma (s ), \gamma^\prime (s ))\dd s
\]
exists and is, actually, $\sfC^\infty $.  Indeed, an absolutely continuous minimizer exists by the 
classical Tonelli's theorem.  By  lower \eqref{eq:lbL0} and upper 
\eqref{eq:UpperFlow} bounds on the Lagrangian, it is easy to understand that minimizers stay
inside ${\rm int} (\Delta_d )$ once $t_0$ and $\delta_0$ are chosen to be appropriately small. 
But then the regularity theory of either \cite{CV} or \cite{AAB} applies and yields Lipschitz 
regularity on $[0,t_0 ]$. Since, the Lagrangian $\cL$ is strictly convex in the second argument,
 and, in the interiour of 
$\Delta_d$, it is  $\sfC^\infty$ in both arguments, the $\sfC^\infty$ of the minimizer follows from
the implicit function theorem, see e.g. \cite{BGH}.

\end{document}

%% file: Macros.tex
\newcommand{\inftwo}[2]{\inf_{\substack{#1 \\ #2}}} 

\newcommand{\cA}{\ensuremath{\mathcal A}}

\newcommand{\cC}{\ensuremath{\mathcal C}}

\newcommand{\cG}{\ensuremath{\mathcal G}}
\newcommand{\cH}{\ensuremath{\mathcal H}}
\newcommand{\cI}{\ensuremath{\mathcal I}}

\newcommand{\cL}{\ensuremath{\mathcal L}}
\newcommand{\cM}{\ensuremath{\mathcal M}}

\newcommand{\cR}{\ensuremath{\mathcal R}}
\newcommand{\cS}{\ensuremath{\mathcal S}}
\newcommand{\cT}{\ensuremath{\mathcal T}}
\newcommand{\cU}{\ensuremath{\mathcal U}}

\newcommand{\frr}{\ensuremath{\mathfrak r}}

\newcommand{\bbE}{{\ensuremath{\mathbb E}} }

\newcommand{\bbP}{{\ensuremath{\mathbb P}} }
\newcommand{\bbQ}{{\ensuremath{\mathbb Q}} }
\newcommand{\bbR}{{\ensuremath{\mathbb R}} }

\newcommand{\bbX}{{\ensuremath{\mathbb X}} }

\newcommand{\bbZ}{{\ensuremath{\mathbb Z}} }

\newcommand{\sfs}{{\sf s}}
\newcommand{\sfr}{{\sf r}}

\newcommand{\sfC}{{\sf C}}
\newcommand{\sfS}{{\sf S}}

\newcommand{\ualpha}{\underline{\alpha}}
\newcommand{\ubeta}{\underline{\beta}}
\newcommand{\ugamma}{\underline{\gamma}}
\newcommand{\utheta}{\underline{\theta}}
\newcommand{\um}{\underline{m}}
\newcommand{\ump}{\underline{m}^\prime}
\newcommand{\uv}{\underline{v}}
\newcommand{\ux}{\underline{x}}
\newcommand{\uy}{\underline{y}}
\newcommand{\uf}{\underline{f}}

\def\1{\ifmmode {1\hskip -3pt \rm{I}}
\else {\hbox {$1\hskip -3pt \rm{I}$}}\fi} 

\newcommand{\smallo}{o}

\newcommand{\soN}{\smallo_N (1)}
\newcommand{\df}{\stackrel{\Delta}{=}}

\newcommand{\sth}{~ :~}

\newcommand{\abs}[1]{\lvert#1\rvert}  

\newcommand{\lb}{\left(}
\newcommand{\rb}{\right)}
\newcommand{\lbr}{\left\{}
\newcommand{\rbr}{\right\}}

\newcommand{\case}[1]{C{\small ASE}~#1\!\! }

\newcommand{\ket}[1]{\vert#1\rangle}
\newcommand{\bra}[1]{\langle#1\vert}

\newcommand{\dd}{{\rm d}}

\newcommand{\ent}{{\rm ent}}

\newcommand{\be}{\begin{equation}}
\newcommand{\ee}{\end{equation}}

%% file: MF-KAM-Fin.bbl
\begin{thebibliography}{10}










\bibitem{A}
N.~ Anantharaman. 
\newblock On the zero-temperature or vanishing viscosity limit for certain Markov processes arising from Lagrangian dynamics. 
\newblock {\em J. Eur.Math.Soc. (JEMS)}   6, 2: 207--276, 2004. 


\bibitem{AKN}
M.~Aizenman, A.~Klein, and C.~Newman.
\newblock Percolation methods for disordered quantum {I}sing models.
\newblock In R.~Kotecky, editor, {\em Phase Transitions: Mathematics, Physics,
  Biology,..}, pages 1--26. World Scientific, Singapore, 1993.

\bibitem{AN94}
M.~Aizenman and B.~Nachtergaele.
\newblock Geometric aspects of quantum spin states.
\newblock {\em Commun. Math. Phys.}, 164:17--63, 1994.

\bibitem{AAB}
L.~Ambrosio, O.~Ascenzi, and G.~Buttazzo.
\newblock Lipschitz regularity for minimizers of integral functionals
with highly discontinuous integrands. 
\newblock {\em J. Math. Anal. Applications}, 142: 301---316, 1989. 




\bibitem{BS12}
V.~Bapst and  G.~Semerjian.
\newblock On quantum mean-filed models and their quantum annealing.
\newblock arXiv:1203.6003v1 2012. 









\bibitem{BEGK}
A.~Bovier,  M.~Eckhoff,  V.~Gayrard, M.~Klein. 
\newblock Metastability and low lying spectra in reversible {M}arkov
              chains.
\newblock {\em Comm. Math. Phys.}, 228, 2: 219--255, 2002.

\bibitem{BGH} 
G. Buttazzo, M. Giaquinta, S. Hildebrandt 
\newblock One-dimensional Variational
Problems. An Introduction.
\newblock {\em Oxford Lecture Series in Mathematics and
its Applications.}, 1998


\bibitem{CKP}
M.~Campanino, A.~Klein, and J.F Perez.
\newblock Localization in the ground state of the {I}sing model with a random
  transverse field.
\newblock {\em Commun. Math. Phys.}, 135: 499--515, 1991.

\bibitem{CDL}
I.~Capuzzo-Dolcetta P-L~Lions. 
\newblock Hamilton-{J}acobi equations with state constraints. 
\newblock {\em Trans. Amer. Math. Soc.}, 318, 2: 643--683, 1990.

\bibitem{CCIL}
L.~Chayes, N.~Crawford, D.~Ioffe, and A.~Levit.
\newblock The phase diagram of the quantum {C}urie-{W}eiss model.
\newblock  {\em J. Stat. Phys.}  133, 1: 131--149, 2008.

\bibitem{CV}
F.H.~Clarke, and R.B.~Vinter. 
\newblock Regularity properties of solutions to the basic 
problem in the calculus of variations. 
\newblock{\em Trans. Amer. Math. Soc,} 289(1): 73--98, 1985. 

\bibitem{CI}
N.~Crawford and D.~ Ioffe. 
\newblock  Random Current Representation for Transverse Field Ising Model.
\newblock  {\em Comm. Math. Phys}  296(2):  447--474, 2010. 

\bibitem{DG}
D.A.~Dawson, and J.~G\:{a}rtner.
\newblock Large deviations from the McKean-Vlasov limit for weakly interacting diffusions. 
\newblock {\em Stochastics}  20, 4: 247--308, 1987.

\bibitem{DS}
M.~Dimassi,  and J.~ Sj\"{o}strand. 
\newblock
{\em 
Spectral asymptotics in the semi-classical limit.}
\newblock 
Cambridge Univ. Press, 
London Mathematical Society Lecture Note Series {\bf 268} 1999. 



\bibitem{DZ}
A.~Dembo,  and O~Zetouni.
\newblock  {\em  Large deviations techniques and applications},
\newblock Stochastic Modeling and Applied Probability, 38. Springer-Verlag, Berlin, 2010


\bibitem{F}
A.~Fathi. 
\newblock {\em Weak KAM Theorem in Lagrangian Dynamics}, Preliminary Version.


\bibitem{Gi69}
J.~Ginibre.
\newblock Existence of phase transitions for quantum lattice systems.
\newblock {\em Commun. Math, Phys.}, 14:205--234, 1969.




\bibitem{GUW}
Ch.~Goldschmidt, D.~ Ueltschi and P.~Windridge. 
\newblock Quantum Heisenberg models and their probabilistic representations.  
\newblock In 
 Entropy and the quantum II, {\em Contemp. Math.},  552:   177--224, Amer. Math. Soc., Providence, RI, 2011, 

\bibitem{H88}
B.~Helffer.
\newblock {\em Semi-classical analysis for the Schr\"{o}dinger Operator and applications.}
LNM 1336, Springer, Heidelberg 1988. 

\bibitem{H02}
B.~Helffer.
\newblock {\em
Semiclassical analysis, Witten Laplacians, and statistical mechanics.}
\newblock 
World Scientific, 
Partial Differential Equations and Applications {\bf 1} 2002. 

\bibitem{I91}
D.~Ioffe 
\newblock On some applicable versions of abstract large deviations theorems. 
\newblock {\em Ann. Probab.}  19, 4: 1629--1639, 1991.


\bibitem{ILN}
D.~Ioffe.
\newblock Stochastic geometry of classical and quantum {I}sing models.
\newblock In {\em Lecture Notes in Mathematics 1970}, 87-126,  Springer, 2009.


\bibitem{KR1}
M.~Klein and E.~Rosenberger.
\newblock
Agmon-type estimates for a class of difference operators. 
\newblock {\em  Ann. Henri Poincare}  9(6): 1177--1215, 2008.

\bibitem{KR2}
M.~Klein and E.~Rosenberger.
\newblock 
Harmonic approximation of difference operators. 
\newblock {\em  J. Funct. Anal.}  257 (11):  3409--3453, 2009. 

\bibitem{KR3}
M.~Klein and E.~Rosenberger.
\newblock Asymptotic eigenfunctions for a class of difference operators.  
\newblock {\em Asymptot. Anal.}  73 (1-2):  1--36, 2011

\bibitem{U}
D.~ Ueltschi. 
\newblock Geometric and probabilistic aspects of boson lattice models.
\newblock   In and out of equilibrium, 
{\em Progr. Probab}, 51:  363--391,  Birkh\"{a}user Boston, Boston, MA, 2002. 




\end{thebibliography}
